\documentclass[a4paper,12pt, twoside, pdftex]{amsart}

\usepackage{fullpage}
\usepackage{amsmath, amsthm, amssymb}
\usepackage{txfonts}
\usepackage{amsfonts}
\usepackage{prettyref}
\usepackage{mathrsfs}
\usepackage[all]{xy}
\usepackage[utf8]{inputenc}

\usepackage{tikz}
\usetikzlibrary{backgrounds}

\usepackage[foot]{amsaddr} 

\theoremstyle{plain}
\newtheorem{thm}{Theorem}[section]
\newtheorem{cor}[thm]{Corollary}
\newtheorem{lem}[thm]{Lemma}
\newtheorem{prop}[thm]{Proposition}
\newtheorem*{acknowledgements}{Acknowledgements}
\newtheorem*{notations}{Notations and conventions}

\theoremstyle{definition}
\newtheorem{dfn}{Definition}[section]
\theoremstyle{remark}
\newtheorem{rmk}{Remark}[section]

\numberwithin{equation}{section}

\newrefformat{th}{Theorem~\ref{#1}}
\newrefformat{cr}{Corollary~\ref{#1}}
\newrefformat{lm}{Lemma~\ref{#1}}
\newrefformat{dl}{Definition-Lemma~\ref{#1}}
\newrefformat{df}{Definition~\ref{#1}}
\newrefformat{cl}{Claim~\ref{#1}}
\newrefformat{sl}{Sublemma~\ref{#1}}
\newrefformat{pr}{Proposition~\ref{#1}}
\newrefformat{cj}{Conjecture~\ref{#1}}
\newrefformat{st}{Step~\ref{#1}}
\newrefformat{sc}{Section~\ref{#1}}
\newrefformat{df}{Definition~\ref{#1}}
\newrefformat{rm}{Remark~\ref{#1}}
\newrefformat{q}{Question~\ref{#1}}
\newrefformat{pb}{Problem~\ref{#1}}
\newrefformat{cd}{Condition~\ref{#1}}
\newrefformat{eg}{Example~\ref{#1}}
\newrefformat{he}{Heore~\ref{#1}}
\newrefformat{fg}{Figure~\ref{#1}}
\newrefformat{tb}{Table~\ref{#1}}
\newrefformat{as}{Assumption~\ref{#1}}

\newcommand{\pref}{\prettyref}
\newcommand{\Art}{\operatorname{\textbf{Art}_{\bfk}}}
\newcommand{\coh}{\operatorname{coh}}
\newcommand{\Coh}{\operatorname{Coh}}
\newcommand{\Coker}{\operatorname{Coker}}
\newcommand{\Cone}{\operatorname{Cone}}
\newcommand{\Def}{\operatorname{Def}}
\newcommand{\Ext}{\operatorname{Ext}}
\newcommand{\Gr}{\operatorname{Gr}}
\newcommand{\Hom}{\operatorname{Hom}}
\newcommand{\id}{\operatorname{id}}

\newcommand{\Ker}{\operatorname{Ker}}
\newcommand{\NS}{\operatorname{NS}}
\newcommand{\Perf}{\operatorname{Perf}}
\newcommand{\Pf}{\operatorname{Pf}}
\newcommand{\Pic}{\operatorname{Pic}}
\newcommand{\pr}{\operatorname{pr}}

\newcommand{\Set}{\operatorname{\textbf{Set}}}
\newcommand{\Spec}{\operatorname{Spec}}
\newcommand{\Spf}{\operatorname{Spf}}
\newcommand{\supp}{\operatorname{supp}}
\newcommand{\Tor}{\operatorname{Tor}}

\newcommand{\cC}{\mathcal{C}}

\newcommand{\cE}{\mathcal{E}}

\newcommand{\cH}{\mathcal{H}}

\newcommand{\cX}{\mathcal{X}}

\newcommand{\bC}{\mathbb{C}}

\newcommand{\bN}{\mathbb{N}}
\newcommand{\bP}{\mathbb{P}}

\newcommand{\bZ}{\mathbb{Z}}

\newcommand{\bfk}{\mathbf{k}}

\newcommand{\scrC}{\mathscr{C}}
\newcommand{\scrD}{\mathscr{D}}
\newcommand{\scrE}{\mathscr{E}}
\newcommand{\scrF}{\mathscr{F}}

\newcommand{\scrO}{\mathscr{O}}

\newcommand{\scrT}{\mathscr{T}}

\newcommand{\frakm}{\mathfrak{m}}

\title{Categorical generic fiber}
\author[H.~Morimura]{Hayato Morimura}
\address{Kavli Institute for the Physics and Mathematics of the Universe (WPI),
University of Tokyo,
5-1-5 Kashiwanoha,
Kashiwa,
Chiba,
277-8583,
Japan.}
\email{hayato.morimura@ipmu.jp}

\date{}
\begin{document}
\maketitle

\begin{abstract}
For flat proper families of algebraic varieties with a smooth fiber,
we describe the abelian category of coherent sheaves on the generic fiber as a Serre quotient.
As an application,
we prove specialization of derived equivalence.
As another application,
we provide new examples of Fourier--Mukai partners via deformations.
\end{abstract}

\section{Introduction}
\subsection{The main result}
The
\emph{categorical general fiber}
was introduced in
\cite{HMS11}
to study the generic categorical behavior of formal deformations of
$K3$
surfaces.
When a formal deformation is effective,
one can consult the generic fiber of an effectivization.
In analytic setting,
Raynaud constructed the generic fiber of formal deformations \cite{Ray}.
Based on his idea,
Huybrechts--Macr\`i--Stellari developed the categorical general fiber,
providing a categorical analogue of the generic fiber
for noneffective formal deformations of
$K3$
surfaces.
As shown in
\cite{HMS09},
it captures the generic categorical behavior of formal deformations.
Below we briefly review their work.

Let
$\cX \to \Spf \bfk \llbracket t \rrbracket$
be a formal deformation of a smooth projective $\bfk$-variety
$X_0$.
Recall that
the
\emph{abelian category of coherent sheaves on the general fiber}
is the Serre quotient
\begin{align*}
\Coh (\cX_{\bfk (( t ))}) \coloneqq \Coh (\cX) / \Coh (\cX)_0,
\end{align*}
where  
$\Coh (\cX)_0 \subset \Coh (\cX)$
is the full abelian subcategory
spanned by coherent $\bfk \llbracket t \rrbracket$-torsion $\scrO_\cX$-modules.
By
\cite[Theorem 3.2]{Miy}
the derived category
$D^b \left( \Coh \left( \cX_{\bfk (( t ))} \right) \right)$
of the Serre quotient
is equivalent to the Verdier quotient
\begin{align*}
D^b (\cX) / D^b_0 (\cX),
\end{align*}
where
$D^b_0 (\cX)$
is the full triangulated subcategory
spanned by complexes with $\bfk \llbracket t \rrbracket$-torsion cohomology.
Huybrechts--Macr\`i--Stellari showed the $\bfk (( t ))$-linear exact equivalence
\begin{align*}
D^b \left( \Coh \left( \cX_{\bfk (( t ))} \right) \right)
\simeq
D^b_{\rm{c}} \left( \rm{Mod} \left( \scrO_\cX \right) \right) / D^b_{\rm{c}_0} \left(\rm{Mod} \left( \scrO_\cX \right) \right)
\end{align*}
when
$X_0$
is a
$K3$
surface
\cite[Theorem 1.1]{HMS11}.
Here 
$D^b_{\rm{c}} \left( \rm{Mod} \left( \scrO_\cX \right) \right)$
is the bounded derived category of $\scrO_\cX$-modules with coherent cohomology
and
$D^b_{\rm{c}_0} \left( \rm{Mod} \left( \scrO_\cX \right) \right)$
is its full triangulated subcategory
spanned by complexes with $\bfk \llbracket t \rrbracket$-torsion cohomology. 
The latter Verdier quotient is called the
\emph{derived category of the general fiber}.

If
$\cX$
is effective with a proper effectivization
$X$,
i.e.,
isomorphic to the formal completion of a proper $\bfk \llbracket t \rrbracket$-scheme
$X$,
then by
\cite[Corollorary III5.1.6]{GD}
we have
\begin{align*}
D^b (X) \simeq D^b (\cX)
\end{align*}
and
$D^b \left( \Coh \left( \cX_{\bfk (( t ))} \right) \right)$
is equivalent to the Verdier quotient
\begin{align*}
D^b (X) / D^b_0 (X),
\end{align*}
which is $\bfk (( t ))$-linear.
This can be regarded as an effectivization of
$D^b \left( \Coh \left( \cX_{\bfk (( t ))} \right) \right)$. On the other hand,
the derived category of the generic fiber
$X_{\bfk (( t ))}$
of the effectivization
$X$
gives another $\bfk (( t ))$-linear category.
By
\cite[Theorem 1.2]{BFN}
also
$D^b (X_{\bfk (( t ))})$
is obtained from the $\bfk \llbracket t \rrbracket$-linear category
$D^b (X)$.
As both
$D^b (X) / D^b_0 (X)$
and
$D^b (X_{\bfk (( t ))})$
carry sufficient information on the generic categorical behavior of formal deformations,
it is natural to ask
whether they are equivalent as a $\bfk (( t ))$-linear triangulated category.
Motivated by this question,
we prove the following.

\begin{thm}[Theorem 2.5] \label{thm:esurjIntro}
Let
$X_R$
be a smooth separated family over a noetherian connected regular affine $\bfk$-scheme
$\Spec R$
whose closed points are
$\bfk$-rational.
Let
$K$
be the field of fractions of
$R$
and
$X_K$
the generic fiber.
Then there exists a $K$-linear equivalence
\begin{align*}
\coh(X_K) \simeq \coh(X_R) / \coh(X_R)_0
\end{align*}
of abelian categories,
where
$\coh(X_R)_0$
is the Serre subcategory spanned by $R$-torsion sheaves.
\end{thm}

\begin{cor}[Corollary 2.6] \label{cor:CGFIntro} 
Under the same assumption as above,
there exists an exact $K$-linear equivalence
\begin{align*}
D^b(X_K) \simeq D^b(X_R) / D^b_0(X_R),
\end{align*}
where
$D^b_0(X_R)$
is the full triangulated subcategory
spanned by complexes with $R$-torsion cohomology.
\end{cor}

We impose the technical condition on
$\Spec R$
to include smooth proper effectivizations of formal families over formal power series rings,
besides smooth proper families over nonsingular affine $\bfk$-varieties.

\subsection{The first application}
Probably the most important advantage to describe the derived category of the generic fiber as a Verdier quotient is that
any object of
$D^b(X_K)$
lifts to that of
$D^b(X_R)$.
The quotient description extends to nonaffine base case for flat proper families of $\bfk$-varieties with a smooth fiber.

\begin{thm} [\pref{thm:gbase}] \label{thm:globalCGFIntro}
Let
$\pi \colon X \to S$
be a flat proper morphism of $\bfk$-varieties.
Assume that
at least one of the closed fibers of
$\pi$
is smooth.
Then there exists a $K$-linear exact equivalence
\begin{align*} 
D^b(X) / \Ker(\bar{\iota}^*_K)
\simeq
D^b(X_K),
\end{align*}
where
$K = k(S)$
is the function field
and
$\bar{\iota}_K \colon X_K \to X$
is the canonical morphism.
\end{thm}

In particular,
there always exists a lift
$\cE \in D^b(X \times_S X^\prime)$
of a Fourier--Mukai kernel
$\cE_K \in D^b(X_K \times X^\prime_K)$
whenever
$X_K, X^\prime_K$
are derived-equivalent.
It allows us to extend the derived equivalence of the generic fibers to that of general fibers.

\begin{thm} [Corollary \pref{cor:1.4}] \label{thm:mainIntro}
Let
$\pi \colon X \to S, \pi^\prime \colon X^\prime \to S$
be flat proper morphisms of $\bfk$-varieties.
Assume that
their generic fibers
$X_K, X^\prime_K$
are derived-equivalent.
Assume further that
at least one of the closed fibers of each
$\pi$
and
$\pi^\prime$
are smooth.
Then there exists a nonempty open subset
$U \subset S$
over which the restrictions
$X_U, X^\prime_U$
become derived-equivalent.
In particular,
over
$U$
any pair of closed fibers are derived-equivalent.
\end{thm}

This can be regarded as a categorical analogue of the fact that
isomorphic generic fibers imply birational families in our setting. 
In classical algebraic geometry,
the generic fiber often controls behaviors of general fibers.
For instance,
if the generic fiber satisfies a certain property
which is constructible,
then general fibers also satisfy the same property.
Such characteristics of the generic fiber must be translated via Gabriel's theorem
\cite{Gab}
into the abelian category of coherent sheaves. 
Nevertheless,
as there exist pairs of nonisomorphic derived-equivalent $K$-varieties,
it makes sense to wonder
how the derived category of the generic fiber affects that of general fibers.

When
$\bfk$
is a universal domain,
i.e.,
an algebraically closed field of infinite transcendence degree of the prime field,
very general fibers are obtained as the base changes of the geometric generic fiber along isomorphisms
$\bfk \cong \bar{K}$
from
\cite[Lemma 2.1]{Via}.
Hence in this case the derived equivalence of the generic fibers induces that of very general fibers.
The proof of
\cite[Theorem 1.1]{Mor}
shows that
\pref{thm:mainIntro}
follows
as soon as a lift of the Fourier--Mukai kernel induces the derived equivalence of a single pair of closed fibers.
However,
since the composition of the canonical morphisms
$R \to K \to \bar{K}$
with a fixed
$\bar{K} \cong \bfk$
does not coincide with the surjection
$R \to \bfk$,
the induced Fourier Mukai kernel on very general fibers are in general different from the restriction of the lift.

Consider another description
\begin{align*}
D^b_{dg}(X_\xi)
\simeq
\Perf_{dg}(X_\xi)
\simeq
\Perf_{dg}(X)
\otimes_{\Perf_{dg}(S)}
\Perf_{dg}(K)
\end{align*}
of a dg enhancement
$D^b_{dg}(X_\xi)$
of
$D^b(X_\xi)$
obtained from
\cite[Theorem 1.2]{BFN}
and
\cite{Coh}.
By
\cite{CP}
and
\cite{Miy}
the category
$\Perf_{dg}(K)$
is a dg enhancement of the Verdier quotient
\begin{align*}
D^b(S) / D^b_{\leq \dim S - 1}(S),
\end{align*}
where
$D^b_{\leq \dim S - 1}(S)$
is the full triangulated subcategory spanned by objects with cohomology supported on dimension at most
$\dim S - 1$.
Thus removing the torsion support from
$D^b(X)$
is equivalent to removing all closed fibers from the supports of objects of
$D^b(X)$.
In particular,
from a collection of finitely many objects and Hom-sets between them in
$D^b(X)$,
one can remove its torsion support by shrinking the base.

In order to show that
the restriction of the lift to general fibers define equivalences,
we apply the argument in the proof of
\cite[Theorem 1.1]{Mor}
to a fixed strong generator
$E_U$
of
$D^b(X_U)$,
which always exists over sufficiently small open subset
$U \subset S$.
By shrinking
$U$
further if necessary,
one can remove the torsion parts with respect to the base from
$E_U$
and
its relevant Hom-sets.
Then we invoke some basic categorical results to show that
the value of the counit morphism on the trimmed strong generator is an isomorphism,
which implies that
the restriction of the counit morphism is a natural isomorphism.

\pref{thm:mainIntro}
tells us that
the derived category of the generic fiber determines an $U$-linear triangulated category 
$D^b(X_U) \simeq D^b(X) / D^b_Z(X)$
for some open subset
$U \subset S$
and
its complement
$Z$,
where
$D^b_Z (X) \subset D^b(X)$
is the full $S$-linear triangulated subcategory with cohomology supported on
$X_Z$.
In general,
the derived category of the generic fiber cannot recover that of the initial fiber
as we have
$D^b(X_K) \simeq D^b(X_R) / D^b(X_0)$
for
$R = \bfk \llbracket t \rrbracket$.
Rather,
\cite{HMS09}
and
our results suggest that
it carries information on derived categories of general fibers.
We expect that
\pref{thm:mainIntro}
provides a way to seek
categorical constructible properties
and
their derived invariance.

\subsection{The second application}
Another advantage to describe the derived category of the generic fiber as a Verdier quotient
is that
Fourier--Mukai machinery carries over easily.
Suppose that
\begin{align*}
\Phi_\cE \colon D^b (X_R) \to D^b (X^\prime_R)
\end{align*}
is a relative Fourier--Mukai transform
of
smooth proper families over
$R$
with kernel
$\cE \in D^b(X_R \times_R X^\prime_R)$.
Then
$\Phi_\cE$
induces a Fourier--Mukai transform
\begin{align*}
\Phi_{\cE_K} \colon D^b (X_K) \to D^b (X^\prime_K)
\end{align*}
of the generic fibers,
where the kernel
$\cE_K$
is the pullback along the canonical inclusion
$R \hookrightarrow K$.
By the standard argument the further base change to the closure 
defines a Fourier--Mukai transform
\begin{align*}
\Phi_{\cE_{\bar{K}}}
\colon
D^b (X_{\bar{K}}) \to D^b (X^\prime_{\bar{K}})
\end{align*}
of the geometric generic fibers.

Typical examples for our results are given by deformations of higher dimensional Calabi--Yau manifolds. 
Recently,
the author proved the following.

\begin{thm}[{\cite[Theorem 1.1]{Mor}}] \label{thm:IDeqIntro}
Let
$X_0, X_0^\prime$
be derived-equivalent Calabi--Yau manifolds
of dimension more than two.
Then there exists a nonsingular affine $\bfk$-variety
$\Spec S$ 
such that
smooth projective versal deformations
$X_S, X^\prime_S$
over
$S$  
are derived-equivalent. 
\end{thm}

Here,
the derived equivalence is given by the relative Fourier--Mukai transform with kernel obtained by deformation of the original kernel for central fibers.
Similarly,
the Fourier--Mukai transform of central fibers extends to proper effectivizations of universal formal families.   
Then our current results implies the derived equivalence of the
generic
and
geometric generic
fibers of both families.
One can check that
they are Calabi--Yau manifolds of dimension
$\dim X_0$.
When the central fibers are nonbirational,
in some special cases
one can deduce the nonbirationality of the
generic
and
geometric generic
fibers.
See also specialization of birational types over a smooth connected curve
\cite[Theorem 1.1]{KT}.
To summarize,
we obtain

\begin{thm}[Theorem 5.7] \label{thm:FMPIntro}
Let
$X_0, X_0^\prime$
be derived-equivalent Calabi--Yau manifolds
of dimension more than two.
Then the generic fibers
$X_K, X^\prime_K$
of proper effectivizations
and that
$X_Q, X^\prime_Q$
of smooth projective versal deformations
are respectively derived-equivalent Calabi--Yau manifolds.
If,
in addition,
we have either
$\NS_{tor} X_0 \neq \NS_{tor} X^\prime_0$,
or
$\rho (X_0) = \rho (X^\prime_0) = 1$
and
$\deg (X_0) \neq \deg (X^\prime_0)$,
then
they are respectively nonbirational.
Similar results hold for the geometric generic fibers
$X_{\bar{K}}, X^\prime_{\bar{K}}$
and
$X_{\bar{Q}}, X^\prime_{\bar{Q}}$.
\end{thm}

Recall that
\emph{Fourier--Mukai partners}
are pairs of nonbirational Calabi--Yau threefolds
that are derived-equivalent.
They are of considerable interest from the viewpoint of
string theory
and
mirror symmetry.
For instance,
the Gross--Popescu pair
\cite{GP, Sch}
and
the Pfaffian--Grassmannian pair
\cite{BC, 0610957}
satisfy
the first
and
the second
conditions in
\pref{thm:FMPIntro}
respectively.
Thus we obtain new examples of Fourier--Mukai partners over the function fields
$K, Q$
and
their closures
$\bar{K}, \bar{Q}$.
Note that
if
$\bfk$
is a universal domain,
then there exists an isomorphism
$\bfk \cong \bar{Q}$
\cite[Lemma 2.1]{Via}.
In particular,
if
$\bfk = \bC$
then the new examples over
$\bar{Q}$
can be regarded as complex manifolds.

When
$X_0, X^\prime_0$
are the Pfaffian--Grassmannian pair,
we demonstrate the subtle difference between
$X_{\bar{Q}}, X^\prime_{\bar{Q}}$
and
known examples.
The geometric generic fibers
$X_{\bar{Q}}, X^\prime_{\bar{Q}}$
are respectively isomorphic to very general fibers as a scheme,
but not
as a $\bC$-variety.
For the same reason,
another known pair
$Y_0, Y^\prime_0$
over
$\bC$
called IMOU varieties cannot be isomorphic to 
$X_{\bar{Q}}, X^\prime_{\bar{Q}}$
as a $\bC$-variety.
It would be interesting
to study the relationship between Fourier--Mukai transforms on very general fibers induced by two different ways,
with
or
without
passing through the geometric generic fibers,
and
to construct nontrivial autoequivalences.

\begin{notations}
We work over an algebraically closed field
$\bfk$
of characteristic
$0$
throughout the paper.
Every time we apply
\cite[Lemma 2.1]{Via}
we always assume
$\bfk$
to be a universal domain without comments.
A $\bfk$-variety is an integral separated $\bfk$-scheme of finite type. 
A Calabi--Yau manifold
$X_0$
is a smooth projective $\bfk$-variety with
trivial canonical bundle
and
$\text{H}^i (X_0, \scrO_{X_0}) = 0$
for
$0 < i < \dim X_0$.
For a noetherian formal scheme
$\cX$
by
$D^b (\cX)$
we denote the bounded derived category of the abelian category
$\Coh (\cX)$
of coherent sheaves on
$\cX$. 
\end{notations}

\begin{acknowledgements}
The author would like to express his gratitude to Kazushi Ueda
for suggesting the problem studied in
\cite{Mor}.
The author would like to thank Paolo Stellari
for answering some questions.
The author also would like to thank
Evgeny Shinder
and
Nicolò Sibilla
for informing him on the paper
\cite{CP}.
Finally,
the author is deeply thankful to anonymous referees
who carefully read the earlier version,
pointed out some mistakes,
and
suggested many useful ideas.
This work was partially supported by
JSPS KAKENHI Grant Number
JP23KJ0341. 
\end{acknowledgements}

\section{The derived category of the generic fiber}
For a scheme
$X_R$
over an integral domain
$R$, 
we have the pullback diagram
\begin{align*}
\begin{gathered}
\xymatrix{
X_K \ar[r]^-{i} \ar_{\pi_K}[d] & X_R \ar^{\pi_R}[d] \\
\Spec K \ar[r]_-{j} & \Spec R ,
}
\end{gathered}
\end{align*}
where
$K$
is the field of fractions of
$R$
and
$X_K$
is the generic fiber,
i.e.,
the fiber over the generic point
$\xi \in \Spec R$.
In the sequel,
we assume the following conditions:
\begin{itemize}
\item[(i)]
$X_R$
is connected
and
smooth separated over
$R$.
\item[(ii)]
$\Spec R$
is a noetherian connected regular affine $\bfk$-scheme
whose closed points are
$\bfk$-rational.
\end{itemize}
The assumption guarantees that
$X_R$
and
$X_K$
are noetherian separated regular.
Indeed,
$\scrO_{X_R, x}$
is regular for every closed point
$x \in X_R$
by
\cite[Tag 031E]{SP}.
Under the assumption also
$X_K$
is connected.
An example of
$X_R$
we have in mind is proper effectivizations of miniversal formal families of a smooth projective $\bfk$-variety 
whose deformations are unobstructed.
Note that
in this case
$\Spec R$
is not of finite type over
$\bfk$.
We impose
(ii)
to include such examples,
besides smooth proper families over nonsingular affine $\bfk$-varieties.


We denote by
$\coh(X_R)_0 \subset \coh(X_R)$
the Serre subcategory spanned by $R$-torsion sheaves,
i.e.,
for each
$\scrF \in \coh(X_R)_0$
there is an element
$r \in R$
such that
$r \scrF = 0$.
We write
$\cC = \coh(X_R) / \coh(X_R)_0$
for the Serre quotient.
The natural projection
$p \colon \coh(X_R) \to \cC$
which sends
$\scrF$ to $\scrF_K$
is known to be exact.
By universality of Serre quotient,
the exact functor
\begin{align*}
(-) \otimes_R K \colon \coh (X_R) \to \coh (X_K)
\end{align*}
induces a unique exact functor
\begin{align*}
\Phi \colon \cC \to \coh (X_K)
\end{align*}
such that
$(-) \otimes_R K = \Phi \circ p$. 
Then $\Phi$ defines the derived functor
\begin{align*}
D^b(\cC) \to D^b(X_K),
\end{align*}
which induces a functor
\begin{align*}
\Psi \colon D^b(X_R) / D^b_0(X_R) \to D^b(X_K)
\end{align*}
via 
\cite[Theorem 3.2]{Miy}.
We show that
$\Phi$
and
$\Psi$
are equivalences.
In particular,
\begin{align} \label{eq:Miyachi}
D^b (\cC) \simeq D^b(X_R) / D^b_0(X_R)
\end{align}
gives an alternative description of
$D^b(X_K)$.

\subsection{$K$-linear categorical quotients}
\begin{lem}[cf. {\cite[Proposition 2.3]{HMS11}}] \label{lem:Klin}
The abelian category
$\cC$
is $K$-linear
and
for all
$\scrE, \scrF \in \coh(X_R)$
the natural projection
$p \colon \coh(X_R) \to \cC$
induces a $K$-linear isomorphism
\begin{align} \label{eq:Klin}
\Hom_{X_R}(\scrE, \scrF) \otimes_R K
\cong
\Hom_{\cC}(\scrE_K, \scrF_K).
\end{align}
\end{lem}
\begin{proof}
One can adapt the proof of
\cite[Proposition 2.3]{HMS11}
in a straightforward way. 
\end{proof}

\begin{lem}[cf. {\cite[Proposition 2.9]{HMS11}}] \label{lem:VKlin}
The triangulated category
$D^b (\cC)$
is $K$-linear
and
for all
$E, F \in D^b(X_R)$
the natural projection
$Q \colon D^b(X_R) \to D^b (\cC)$
induces a $K$-linear isomorphism
\begin{align*}
\Hom_{D^b(X_R)} \left( E, F \right) \otimes_R K
\cong
\Hom_{D^b (\cC)} \left( E_K, F_K \right).
\end{align*}
\end{lem}
\begin{proof}
Any morphism in
$D^b (\cC)$
can be represented by a morphism of bounded complexes of objects in
$\cC$,
which is a collection of morphisms in
$\cC$
compatible with the differentials.
Since by
\pref{lem:Klin}
both
the morphisms
and
the differentials
are $K$-linear,
the representative is also $K$-linear.
The $K$-linear isomorphism is a direct consequence of Corollary
\pref{cor:CGF},
whose proof does not rely on it.
\end{proof}

\begin{rmk}
In
\cite{HMS11}
only the case where
$R = \bfk \llbracket t \rrbracket$
was treated.
This is because some results there require
$R$
to be a DVR.
Throughout the paper,
we are free from the requirement and results which rely on it.
\end{rmk}

\subsection{Canonical functor from the Serre quotient}
\begin{prop} \label{prop:GFeq}
The functor
$\Phi \colon \cC \to \coh (X_K)$
is fully faithful.
\end{prop}
\begin{proof}
The images of
$\scrE_K, \scrF_K \in \cC$
by
$\Phi$
are respectively isomorphic to the pullbacks
$i^* \scrE, i^* \scrF$
of some coherent sheaves
$\scrE, \scrF$
on
$X_R$.
We have
\begin{align*}
\Hom_{X_K} \left( \Phi( \scrE_K), \Phi (\scrF_K \right)
&=
\Hom_{X_K} \left( i^* \scrE, i^* \scrF \right) \\
&\cong
\Gamma \circ i^* \underline{\Hom}_{X_R} \left( \scrE, \scrF \right) \\
&\cong
j^* \circ (\pi_R)_* \underline{\Hom}_{X_R} \left( \scrE, \scrF \right) \\
&\cong
\Hom_{X_R} \left( \scrE, \scrF \right) \otimes_R K \\
&\cong
\Hom_{\cC} \left( \scrE_K, \scrF_K \right),
\end{align*}
where the second, the third, and the fourth isomorphisms follow from 
flat base change,
\pref{lem:stalk}
below,
and
\pref{lem:Klin}
respectively.
\end{proof}

\begin{lem} \label{lem:stalk}
For all
$\scrE, \scrF \in \coh (X_R)$
we have an isomorphism
\begin{align*}
(\pi_R)_* \underline{\Hom}_{X_R} \left( \scrE, \scrF \right) \otimes_R K
\cong
\Hom_{X_R} \left( \scrE, \scrF \right) \otimes_R K .
\end{align*}
\end{lem}
\begin{proof}
We may consider
$(\pi_R)_* \underline{\Hom}_{X_R} \left( \scrE, \scrF \right) \otimes_R K$
as the stalk of the sheaf
$(\pi_R)_* \underline{\Hom}_{X_R} \left( \scrE, \scrF \right)$
at the generic point
$\xi$
of
$\Spec R$.
For an affine open cover
$\Spec R = \{ D(f) \}, \ 0 \neq f \in R$,
take any germ
$\langle D(f), s \rangle$ of $(\pi_R)_* \underline{\Hom}_{X_R} \left( \scrE, \scrF \right)_\xi$.
Since
$(\pi_R)_* \underline{\Hom}_{X_R} \left( \scrE, \scrF \right)$
is a quasi-coherent sheaf on an affine scheme
as
$X_R$
is noetherian,
by
\cite[Lemma II5.3]{Har}
there exists an integer
$n \geq 0$
such that
$f^n s$
becomes a global section.
Let 
\begin{align*}
\phi \colon (\pi_R)_* \underline{\Hom}_{X_R} \left( \scrE, \scrF \right)_\xi
\to
\Hom_{X_R} \left( \scrE, \scrF \right) \otimes_R K
\end{align*}
be the homomorphism of $R$-algebras
which sends
$\langle D(f), s \rangle$
to
$f^m s \otimes (1/f^m)$,
where
$m$
is the minimum integer
such that
$f^m s$
becomes a global section. 
One can check that this is well-defined.
The inverse
$\phi^{-1}$
is given by the map
which sends
$v \otimes(g/f)$
to
$\langle D(f), (gv/f) \rangle$
for
$v \in \Hom_{X_R} \left( \scrE, \scrF \right)$
and
$g \in R$. 
\end{proof}

\begin{thm} \label{thm:esurj}
The functor
$\Phi \colon \cC \to \coh (X_K)$
is a $K$-linear equivalence of abelian categories.
\end{thm}
\begin{proof}
It suffices to show that
$\Phi$
is essentially surjective.
By assumption
$X_K$
is connected.
Let
$\scrF_\xi$
be an object of
$\coh (X_K)$.
Since
$X_K$
is noetherian integral separated regular,
by
\cite[Exercise III6.8]{Har}
any coherent sheaf on
$X_K$
can be obtained as the cokernel of a morphism of locally free sheaves of finite rank.
The essential image of
$\Phi$
is a full abelian subcategory of
$\coh (X_K)$.
In particular,
it is closed under taking cokernels.
Hence we may assume
$\scrF_\xi$
to be a locally free sheaf of finite rank.

Take an affine open cover
$\{ U_i \} ^m_{i=1}, U_i = \Spec A_i$
of
$X_R$
such that
the restriction of
$\scrF_\xi$
to each affine open subset
$V_i = U_i \times_R K$
of
$X_K$
is isomorphic to a finite rank free
$\tilde{B}_i = \tilde{A}_i \otimes_R K$-module  
\begin{align*}
F_i 
=
\tilde{B}^{\oplus N}_i.
\end{align*}  
Let
$\phi_{ij}
=
\phi_i \circ \phi^{-1}_j
\colon
F_j |_{V_{ij}}
\to
F_i |_{V_{ij}}$
be isomorphisms on
$V_{ij} = V_i \cap V_j$
where
$\phi_i \colon F_\xi |_{V_i} \to F_i$
are trivializations with their inverses
$\phi^{-1}_i \colon F_i \to F_\xi |_{V_i}$.
In other words,
we have the commutative diagrams
\begin{align*}
\begin{gathered}
\xymatrix{
F_\xi |_{V_{ij}} \ar@{=}[r] & F_\xi |_{V_{ij}} \ar^{\phi_i}[d] \\
F_j |_{V_{ij}} \ar[r]_-{\phi_{ij}} \ar^{\phi^{-1}_j}[u] & F_i |_{V_{ij}}.
}
\end{gathered}
\end{align*}

From
$F_i$
we obtain a rank
$N$
free $\tilde{A}_i$-module
\begin{align*}
E_i
=
\tilde{A}^{\oplus N}_i
\end{align*}  
with the same generators.
By construction tensoring
$K$
with
$E_i$
recovers
$F_i$.
Now,
up to shrinking the base
$\Spec R$,
we glue
$E_i$
to construct a coherent sheaf 
$\scrE$
on
$X_R$
such that
$\scrE \otimes_R K \cong \scrF_\xi$.
By
\pref{lem:Klin}
there are lifts 
$\bar{\phi}_{ij} \colon E_j |_{U_{ij}} \to E_i |_{U_{ij}}$
on
$U_{ij} = U_i \cap U_j$
of
$\phi_{ij}$
along
\pref{eq:Klin}.
Namely,
we have
\begin{align*}
\bar{\phi}_{ij} \otimes_R 1 / r_{ij}
=
\phi_{ij}
\end{align*}
for some
$r_{ij} \in R$.

Consider the affine open subset
\begin{align*}
\Spec T
\subset
\Spec R
\end{align*}
defined by
$r_{ij} \neq 0, 1 \leq i, j \leq m$.
On the base changes
$U_{ij, T} = U_{ij} \times_R T$
all
$r_{ij}$
become invertible.
Hence
$\phi_{ij}$
canonically lift to isomorphisms 
\begin{align*}
r^{-1}_{ij} \bar{\phi}_{ij}
\colon
E_j |_{U_{ij, T}}
\to
E_i |_{U_{ij, T}}
\end{align*}
which injectively map to
$r^{-1}_{ij} \bar{\phi}_{ij} \otimes_T 1 = \phi_{ij}$
under
\begin{align*}
\Hom_{U_{ij, T}}(E_j |_{U_{ij, T}}, E_i |_{U_{ij, T}})
\xrightarrow{- \otimes_T K}
\Hom_{U_{ij, T}}(E_j |_{U_{ij, T}}, E_i |_{U_{ij, T}}) \otimes_T K
\cong
\Hom_{\cC}((E_j)_K, (E_j)_K).
\end{align*}
Clearly,
the lifts satisfy the cocycle condition.
Thus
$E_i |_{U_{i, T}}$
glue to yield a locally free sheaf
$\tilde{\scrE}$
on
$X_T = X_R \times_R T$
such that
$\tilde{\scrE} \otimes_T K \cong \scrF_\xi$.

By
\cite[Exercise II5.15]{Har}
the lift
$\tilde{\scrE}$
extends to a coherent sheaf
$\scrE$
on
$X_R$.
Since the exact functor
$(-) \otimes_R K$
factorizes through
\begin{align*}
\coh(X_R)
\to
\coh(X_T)
\to
\coh(X_K)
\end{align*}
and
it sends
$\scrE$
to
$\scrF_\xi$,
there is an object
$\scrE_K \in \cC$
which maps to
$\scrF_\xi$
under
$\Phi$.
\end{proof}

\subsection{Canonical functor from the Verdier quotient}
As the functor
$\Phi \colon \cC \to \coh (X_K)$ 
is exact,
termwise application of
$\Phi$
defines a derived functor
$D^b (\cC) \to D^b(X_K)$.
By universality of Verdier quotient,
the induced functor
\begin{align*}
\Psi \colon D^b(X_R) / D^b_0(X_R) \to D^b(X_K)
\end{align*}
by
\pref{eq:Miyachi}
coincides with
$D^b (\cC) \to D^b(X_K)$.
From
\pref{thm:esurj}
we obtain

\begin{cor} \label{cor:CGF}
The functor
$\Psi \colon D^b(X_R) / D^b_0(X_R) \to D^b(X_K)$
is a $K$-linear exact equivalence.
\end{cor}

In Section
$5$
we will extend
\pref{thm:esurj}
and
Corollary
\pref{cor:CGF}
to nonaffine base case
for flat proper families of $\bfk$-varieties a smooth fiber.

\section{Comparison with the categorical general fiber}
\subsection{The abelian category of coherent sheaves on the general fiber}
Consider the abelian category of coherent sheaves on the general fiber
$\Coh (\cX_{\bfk (( t ))})
\coloneqq
\Coh (\cX) / \Coh (\cX)_0$
for a formal deformation
$\cX$
of a smooth projective $\bfk$-variety over
$\bfk \llbracket t \rrbracket$.
In the case
where
$\cX$
is effective with a proper effectivization,
one can obtain 
$\Coh (\cX_{\bfk (( t ))})$
via formal completion along the closed fiber in the following sense.

\begin{cor} \label{cor:effectivize}
Let
$\cX = \hat{X}_R \to \Spf R$
be an effective formal deformation of a smooth projective variety with a proper effectivization
$X_R$.
Then the abelian category of coherent sheaves on the general fiber of
$\cX$
is equivalent to that on the generic fiber
$X_K$
of its effectivization, 
i.e.,
there exists a $K$-linear equivalence 
\begin{align*}
\coh(X_K)
\simeq
\Coh(\cX_K)
\end{align*}
of abelian categories.
\end{cor}
\begin{proof}
We have the pullback diagram
\begin{align*}
\begin{gathered}
\xymatrix{
\cX = \hat{X}_R \ar[r]^-{\iota} \ar_{\hat{\pi}}[d]
& X_R \ar^{\pi_R}[d] \\
\Spf R \ar[r]^-{}
& \Spec R
}
\end{gathered}
\end{align*}
of noetherian formal schemes.
Since $R$ is a complete local noetherian ring,
one can apply
\cite[Corollorary III5.1.6]{GD}
to see that
the functor
\begin{align} \label{eq:Alg}
\coh (X_R) \to \Coh (\cX),
\end{align}
which sends each coherent sheaf
$\scrF$
on
$X_R$
to its formal completion
$\hat{\scrF}$
along the closed fiber,
is an $R$-linear equivalence of abelian categories.
By universality of Serre quotient,
we obtain the induced $K$-linear equivalence 
\begin{align*}
\coh(X_R) / \coh(X_R)_0 \to \Coh(\cX) / \Coh(\cX)_0.
\end{align*}
\end{proof}

\subsection{Serre functor}
In the case
where
$\cX$
is effective with a proper effectivization,
the Serre functor from
\cite[Theorem 1.1]{HMS11}
constructed
when
$\cX$
is a formal deformation of a
$K3$
surface,
extends to
smooth projective varieties
and
formal power series rings of any finite dimension
in a straightforward way.

\begin{prop}
Let
$\cX = \hat{X}_R \to \Spf R$
be an effective formal deformation of a $d$-dimensional smooth projective variety with a proper effectivization
$X_R$.
Then a Serre functor on
$D^b(\cX_K)$
is given by
\begin{align*}
S (\hat{E}_K) = ( \widehat{E \otimes \omega_{\pi_R}})_K [d],
\end{align*}
where
$\omega_{\pi_R}$
is the dualizing sheaf for
$\pi_R$.
\end{prop} 
\begin{proof}
We have
\begin{align*}
\Hom_{D^b \left( \Coh \left( \cX_K \right) \right)} \left( \hat{E}_K, \hat{F}_K \right) 
&\cong
\Hom_{D^b(\cX)} \left( \hat{E}, \hat{F} \right) \otimes_R K \\
&\cong
\Hom_{D^b(X_R)} \left( E, F \right) \otimes_R K \\
&\cong
\Hom_{D^b(X_R)} \left( F, E \otimes \omega_{\pi_R} [d]  \right)^\vee \otimes_R K \\
&\cong
\Hom_{D^b(\cX)} \left( \hat{F}, \widehat{{E} \otimes \omega_{\pi_R}} [d]  \right)^\vee \otimes_R K \\
&\cong
\Hom_{D^b \left( \Coh \left( \cX_K \right) \right)} \left( \hat{F}_K, ( \widehat{E \otimes \omega_{\pi_R}} )_K [d]  \right)^\vee,
\end{align*}
where the first and the fifth,
the second and the fourth,
and the third isomorphisms follow from
\pref{lem:VKlin},
the equivalence
\pref{eq:Alg},
and
Serre duality for the smooth morphism
$\pi_R$
of relative dimension
$d$
respectively.
\end{proof}

\subsection{The derived category of the general fiber}
By
\cite[Theorem 1.1]{HMS11}
we have
\begin{align*} 
D^b_{\rm{c}} \left( \rm{Mod} \left( \scrO_\cX \right) \right) / D^b_{\rm{c}_0} \left( \rm{Mod} \left( \scrO_\cX \right) \right)
\simeq
D^b \left( \Coh \left( \cX_{\bfk (( t ))} \right) \right)
\end{align*}
when
$\cX$
is a formal deformation of a K3 surface.
This is deduced from the intermediate $\bfk (( t ))$-linear exact equivalence
\begin{align} \label{eq:connector}
D^b_{\rm{c}} \left( \rm{Mod} \left( \scrO_\cX \right) \right) / D^b_{\rm{c}_0} \left( \rm{Mod} \left( \scrO_\cX \right) \right)
\simeq
D^b(\cX) / D^b_0(\cX)
\end{align}
established in the proof of
\cite[Proposition 3.10]{HMS11}.
While we have
\begin{align*}
D^b_{\rm{c}_0} \left( \rm{Mod} \left( \scrO_\cX \right) \right)
\simeq
D^b_0(\cX)
\end{align*}
by
\cite[Proposition 2.5]{HMS11},
in general the natural inclusion
\begin{align*}
D^b(\cX)
\hookrightarrow
D^b_{\rm{c}} \left( \rm{Mod} \left( \scrO_\cX \right) \right)
\end{align*}
is not an equivalence.
Hence one cannot expect
\pref{eq:connector}
to hold for more general
$\cX$.
However,
in the case 
where
$\cX$
is effective with a proper effectivization,
we have
\begin{align*}
D^b_{\rm{c}} \left( \rm{Mod} \left( \scrO_{X_R} \right) \right) / D^b_{\rm{c}_0} \left( \rm{Mod} \left( \scrO_{X_R} \right) \right)
\simeq
D^b (X_R) / D^b_0(X_R)
\simeq
D^b(X_K)
\end{align*}
by Corollary
\pref{cor:CGF}.
Note that
the first equivalence follows from
\begin{align*}
D^b_{\rm{c}} \left( \rm{Mod} \left( \scrO_{X_R} \right) \right)
\simeq
D^b(X_R),
\end{align*}
which holds for noetherian schemes.
Unless the closed fiber of
$X_R$
is a $K3$ surface,
in general one can only recover the part
\begin{align*}
D^b(\cX) / D^b_0(\cX)
\simeq
D^b(\cX_K)
\end{align*}
via
\cite[Theorem 3.2]{Miy}
and
formal completion along the closed fiber in the sense of Corollary
\pref{cor:effectivize}.

\section{Induced Fourier--Mukai transforms}
As mentioned in
\cite{HMS11},
one advantage to describe the derived category of the generic fiber as a Verdier quotient
is that
the Fourier--Mukai machinery carries over easily.
Given a relative integral functor
$\Phi_\cE \colon D^b(X_R) \to D^b(X^\prime_R)$
on smooth proper families
$\pi_R \colon X_R \to \Spec R$
and
$\pi^\prime_R \colon X^\prime_R \to \Spec R$,
we study the induced derived equivalence on
thier generic fibers,
geometric generic fibers,
and
formal completions.
One will see that
$\Phi_\cE$
being equivalences
when restricted to general fibers implies the derived equivalence of
their generic
and
geometric generic fibers.
We will discuss the opposite direction in Section
$5$
below.
 
\subsection{Induced Functor from smooth proper families to generic fibers} 
\begin{prop} \label{prop:iDeq2}
Let
$X_R, X^\prime_R$
be smooth proper families over
$R$.
If
$\Phi_\cE \colon D^b (X_R) \to D^b (X^\prime_R)$
is a relative Fourier--Mukai functor,
then the induced integral functor
$\Phi_{\cE_K} \colon D^b (X_K) \to D^b (X^\prime_K)$
is an equivalence.
Here,
$\cE_K
\in
D^b (X_R \times_R X^\prime_R) / D^b_0 (X_R \times_R X^\prime_R)$
is the image of
$\cE$
by the natural projection. 
\end{prop}
\begin{proof}
Since objects of
$D^b ( X_R \times_R X^\prime_R ) / D^b_0 ( X_R \times_R X^\prime_R )$ 
are the same as those of
$D^b (X_R \times_R X^\prime_R)$
\cite[Appendix]{HMS11},
the $R$-linear functor
$\Phi_\cE$
induces an integral functor
\begin{align*}
\Phi_{\cE_K}
\colon
D^b(X_R) / D^b_0(X_R) \to D^b(X^\prime_R) / D^b_0(X^\prime_R).
\end{align*}
By Corollary
\pref{cor:CGF}
we have the commutative diagram
\begin{align*}
\begin{gathered}
\xymatrix{
D^b (X_R) \ar[r]^-{\Phi_\cE} \ar_{Q}[d]
& D^b (X^\prime_R) \ar^{Q}[d] \\
D^b (X_K) \ar[r]^-{\Phi_{\cE_K}}
& D^b (X^\prime_K).
}
\end{gathered}
\end{align*}

The inverse functor
$\Phi^{-1}_\cE$
is a left adjoint to
$\Phi_\cE$
as
$\Phi_\cE$
is an equivalence. 
On the other hand,
due to the Grothendieck--Verdier duality
$\Phi_\cE$
has a left adjoint
$\Phi_{\cE_L}$
with
$\cE_L$
a perfect complex on 
$X_R \times_R X^\prime_R$.
By uniqueness of left adjoint up to isomorphism,
it follows
$\Phi^{-1}_\cE \cong \Phi_{\cE_L}$.
Then 
$\Phi^{-1}_\cE$
induces an integral functor
$\Phi_{(\cE_L)_K}$
and
we obtain natural isomorphisms
$\Phi_{(\cE_L)_K} \circ \Phi_{\cE_K}
\cong
\rm{Id}_{D^b (X_K)}, \
\Phi_{\cE_K} \circ \Phi_{(\cE_L)_K}
\cong
\rm{Id}_{D^b (X^\prime_K)}$.
Thus the functor
$\Phi_{\cE_K}$
is an equivalence.
\end{proof}

\begin{rmk}
By universality of Verdier quotient,
Corollary
\pref{cor:CGF}
induces a mere $K$-linear equivalence
$D^b(X_K) \simeq D^b(X^\prime_K)$,
while Proposition
\pref{prop:iDeq2}
preserves Fourier--Mukai kernels.
\end{rmk}

\begin{cor} 
If the restrictions of
$\Phi_\cE$
to general fibers are equivalences, 
then the induced integral functor
$\Phi_{\cE_K} \colon D^b (X_K) \to D^b (X^\prime_K)$
is an equivalence.
\end{cor}
\begin{proof}
We are given a Zariski open subset
$U \subset \Spec R$
such that
any pair of closed fibers of
$X_R$
and
$X^\prime_R$
over
$U$
are derived-equivalent.
By
\cite[Proposition 2.15]{HLS}
the relative integral functor
$\Phi_{\cE_U}$
is an equivalence,
where
$\cE_U$
denotes the restriction of
$\cE$
to
$\pr_1^{-1} \circ \pi^{-1}_R (U)$.
Note that
we do not need the assumption on
$X_U = X_R \times_R U$
to be locally projective in the statement of
\cite[Proposition 2.15]{HLS},
which guarantees the existence of a right adjoint to
$\Phi_{\cE_U}$,
as 
$\pi_R, \pi^\prime_R$
are smooth proper.
Now,
the claim follows immediately from
Proposition \pref{prop:iDeq2}.
\end{proof}

\begin{rmk}
The previous corollary provides from general to generic induction of Fourier--Mukai transforms.
Conversely,
by Corollary
\pref{cor:CGF}
any integral functor
$\Phi_{\cE_K} \colon D^b (X_\xi) \to D^b (X^\prime_\xi)$
lifts to a relative integral functor
$\Phi_\cE \colon D^b (X_R) \to D^b (X^\prime_R)$.
In Section
$5$
we will prove that
the induced functor is an equivalence
when restricted to general fibers.
\end{rmk}

\subsection{Induced Functor from generic to geometric generic fibers}
Due to
\cite{Ola},
one can slightly improve
\cite[Exercise 5.18]{Huy}
and
a well-known fact about the relation between
field extensions
and
Fourier--Mukai transforms.

\begin{lem} \label{lem:cdn}
Let
$X_K, X^\prime_K$
be smooth proper $K$-varieties
and
$\Phi_{\cE_K} \colon D^b(X_K) \to D^b(X^\prime_K)$
an integral functor.
Then
$\Phi_{\cE_K}$
is an equivalence
if and only if
there are isomorphisms
\begin{align} \label{eq:isoms}
\cE_K * (\cE_K)_L \cong \scrO_\Delta, \
(\cE_K)_L * \cE_K \cong \scrO_{\Delta^\prime},
\end{align}
where
$\Delta
\colon
X_K \hookrightarrow X_K \times X_K,
\Delta^\prime
\colon
X^\prime_K \hookrightarrow X^\prime_K \times X^\prime_K$
are the diagonal embeddings.
\end{lem}
\begin{proof}
Assume that
we are given the isomorphisms
\pref{eq:isoms}.
Regarding the isomorphic objects as kernels,
we obtain natural isomorphisms
\begin{align} \label{eq:nisoms}
\Phi_{(\cE_K)_L} \circ \Phi_{\cE_K}
\cong
\rm{Id}_{D^b (X_K)}, \
\Phi_{\cE_K} \circ \Phi_{(\cE_K)_L}
\cong
\rm{Id}_{D^b (X^\prime_K)}.
\end{align}
Thus the functor
$\Phi_{\cE_K}$
is an equivalence.

Conversely,
assume that
$\Phi_{\cE_K}$
is an equivalence.
Let
$(\Phi_{\cE_K})^{-1}$
be its inverse.
Then we have natural isomorphisms
\begin{align*}
(\Phi_{\cE_K})^{-1} \circ \Phi_{\cE_K}
\cong
\rm{Id}_{D^b (X_K)}, \
\Phi_{\cE_K} \circ (\Phi_{\cE_K})^{-1}
\cong
\rm{Id}_{D^b (X^\prime_K)}.
\end{align*}
In particular,
it follows that
$(\Phi_{\cE_K})^{-1}$
is a left adjoint of
$\Phi_{\cE_K}$.
By uniqueness of left adjoint up to isomorphism,
we obtain
$(\Phi_{\cE_K})^{-1} \cong \Phi_{(\cE_K)_L}$
and
\pref{eq:nisoms}.
Thus two pairs of the kernel
$\cE_K * (\cE_K)_L, \scrO_\Delta$
and
$(\cE_K)_L * \cE_K, \scrO_{\Delta^\prime}$
respectively define the same derived autoequivalence of
$D^b(X_K)$
and
$D^b(X^\prime_K)$.
Since any derived equivalence of smooth proper varieties is defined by a unique Fourier--Mukai kernel up to isomorphism
\cite{Ola},
we obtain 
\pref{eq:isoms}.
\end{proof}

\begin{cor} \label{cor:iDeq3}
Let
$X_K, X^\prime_K$
be smooth proper $K$-varieties.
If
$X_K, X^\prime_K$
are derived-equivalent,
then for any field extension
$L_0 / K$
the base changes
$X_{L_0}, X^\prime_{L_0}$
are derived-equivalent.
\end{cor}
\begin{proof}
Let
$\cE_K \in D^b(X_K \times X^\prime_K)$
be a Fourier--Mukai kernel,
which is unique up to isomorphism. 
By
\pref{lem:cdn}
we have isomorphisms
\begin{align*} 
\cE_K * (\cE_K)_L \cong \scrO_\Delta, \
(\cE_K)_L * \cE_K \cong \scrO_{\Delta^\prime}.
\end{align*}
As
$L_0$
is a flat $K$-module,
the pullback by
$\bar{i}^{\prime \prime}
\colon
X_{L_0} \times X_{L_0} 
\to
X_K \times X_K$
yields
\begin{align*} 
\cE_{L_0} * (\cE_{L_0})_L
\cong
\bar{i}^{\prime \prime *} \left( \cE_K * (\cE_K)_L \right)
\cong
\bar{i}^{\prime \prime *} \scrO_\Delta
\cong
\scrO_{\bar{\Delta}},
\end{align*}
where
$\cE_{L_0} = \bar{i}^{\prime \prime *} \cE_K$
and
$\bar{\Delta}
\colon
X_{L_0} \hookrightarrow X_{L_0} \times X_{L_0}$
is the diagonal embedding.
Here,
we have used
$\Omega_{X_{L_0} / L_0}
\cong
\Omega_{X_K / K} \otimes_K L_0$
\cite[Proposition II8.10]{Har}.
Similarly,
we have
$(\cE_{L_0})_L * \cE_{L_0}
\cong
\scrO_{\bar{\Delta^\prime}}$.
Again,
by
\pref{lem:cdn}
we conclude that
$\Phi_{\cE_{L_0}}$
is an equivalence.
\end{proof}

\begin{rmk}
When
$\Spec R$
is an affine $\bfk$-variety
and
$\bfk$
is a universal domain,
very general fibers of
$X_R, X^\prime_R$
are dereived-equivalent
if and only if
so are their geometric generic fibers.
Indeed,
by
\cite[Lemma 2.1]{Via}
there is an isomorphism
$\bfk \to \bar{K}$
along which the pullback of
$X_\bfk \times X^\prime_\bfk$
is isomorphic to
$X_{\bar{K}} \times X^\prime_{\bar{K}}$.
Here,
$X_\bfk, X^\prime_\bfk$
are very general fibers of
$X_R, X^\prime_R$.
One can apply the same argument as in the proof of Corollary
\pref{cor:iDeq3}. 
Note that
we assume
$\bfk$
to be a universal domain to apply
\cite[Lemma 2.1]{Via}.
\end{rmk}

\subsection{Induced Functors to effective formal families and their categorical general fibers}
Assume that
the families 
$\pi_R \colon X_R \to \Spec R$,
$\pi^\prime_R \colon X^\prime_R \to \Spec R$
are effectivizations of formal deformations
$\cX$,
$\cX^\prime$
over
$R$
of smooth projective varieties
$X_0, X^\prime_0$
respectively.
Here,
we will show that
the induced Fourier--Mukai functor to their generic fibers is compatible with formal completion along the closed fibers.
The schemes
$X_R, X^\prime_R$,
their restrictions to the $n$-th order thickenings,
and
their formal completions along the closed fibers
form the commutative diagram  
\begin{align*} 
\begin{gathered}
\xymatrix{
X_n \ar@{^(->}[d]^{\tau_n} & X_n \times_{R_n} X^\prime_n \ar[l]_{q_n} \ar@{^(->}[d]^{\tau^{\prime \prime}_n} \ar[r]^{p_n} & X^\prime_n \ar@{^(->}[d]^{\tau^\prime_n} \\
\cX \ar@{^(->}[d]^\iota & \cX \times_R \cX^\prime \ar[l]_{\hat{q}} \ar@{^(->}[d]^{\iota^{\prime \prime}} \ar[r]^{\hat{p}} & \cX^\prime \ar@{^(->}[d]^{\iota^\prime} \\
X_R & X_R \times_R X^\prime_R \ar[l]_{q} \ar[r]^{p} & X^\prime_R \\
}
\end{gathered}
\end{align*}
of noetherian formal schemes.
Here,
$\hat{q}, \hat{p}$
are canonically determined as the limit by the compatible collections of morphisms
$q_n, p_n$
of schemes,
and
compositions of two sequential vertical arrows give the canonical factorizations of the closed embeddings
\begin{align*}
\kappa_n = \iota \circ \tau_n, \ \ 
\kappa^{\prime \prime}_n = \iota^{\prime \prime} \circ \tau^{\prime \prime}_n, \ \
\kappa^\prime_n = \iota^\prime \circ \tau^\prime_n. 
\end{align*}

\begin{prop} \label{prop:iDeq4}
Given
$\cE \in D^b(X_R \times_R X^\prime_R)$,
the formal completion
$\hat{\cE}$
along the closed fiber
$X_0 \times X^\prime_0$
defines a relative integral functor
\begin{align*}
\Phi_{\hat{\cE}}
=
R \hat{p}_* (\hat{\cE} \otimes^L \hat{q}^* (-))
\colon
D^b (\cX)
\to
D^b (\cX^\prime),
\end{align*}
which makes the diagram
\begin{align*}
\begin{gathered}
\xymatrix{
D^b(X_R) \ar[r]^-{\Phi_\cE} \ar_{G}[d]
& D^b(X^\prime_R) \ar^{G}[d] \\
D^b(\cX) \ar[r]^-{\Phi_{\hat{\cE}}}
& D^b(\cX^\prime)
}
\end{gathered}
\end{align*}
$2$-commutative.
\end{prop}
\begin{proof}
Since objects of
$D^b (\cX^\prime), D^b(\cX \times_R \cX^\prime)$
are quasi-isomorphic to perfect complexes,
the functors
$\hat{q}^* \colon D^b(\cX) \to D^b(\cX \times_R \cX^\prime)$
and
$\hat{\cE} \otimes^L (-)
\colon
D^b(\cX \times_R \cX^\prime)
\to
D^b(\cX \times_R \cX^\prime)$
can be computed by termwise application
after replacing objects with perfect complexes.  
It is known that
$R \hat{p}_*$
is well-defined
and
sends bounded complexes to bounded complexes
by the comparison theorem
\cite[Corollary 4.1.7]{GD}
and
Leray spectral sequence.
Since by the equivalence
\pref{eq:Alg}
any object of
$D^b (\cX)$
can be written as
$\hat{E} = G(E) \cong L \iota^* E$
for some object
$E \in D^b (X_R)$,
we have
\begin{align*}
\Phi_{\hat{\cE}} (\hat{E}) \otimes^L_R R_n
&=
L \tau^{\prime *}_n R \hat{p}_* (\hat{\cE} \otimes^L_R L \hat{q}^* \hat{E}) \\
&\cong
R p_{n *} L \tau^{\prime \prime *}_n (\hat{\cE} \otimes^L_R L \hat{q}^* \hat{E}) \\
&\cong
R p_{n *} (\cE_n \otimes_{R_n} q^*_n E_n) \\
&\cong
\Phi_{\cE_n}(E_n),
\end{align*}
where
$R_n = R / \frakm^{n + 1}_R$,
$E_n = \tau^*_n \hat{E}$,
and
the first isomorphism follows from the comparison theorem.
We also have
\begin{align*}
\Phi_\cE (E) \otimes^L_R R_n
&=
L \kappa^{\prime *}_n R p_* (\cE \otimes^L_R L q^* E) \\
&\cong
R p_{n *} L \kappa^{\prime \prime *}_n (\cE \otimes^L_R L q^* E) \\
&\cong
R p_{n *} (\cE_n \otimes_{R_n} q^*_n E_n) \\
&\cong
\Phi_{\cE_n}(E_n).
\end{align*}
Thus we obtain isomorphisms
\begin{align*}
f_n
\colon
\Phi_{\hat{\cE}} (\hat{E}) \otimes_R R_n
\to
\Phi_\cE (E) \otimes_R R_n
\end{align*}
for any positive integer
$n$
satisfying
$f_{n + 1} \otimes^L_{R_{n + 1}} \id_{R_n} = f_n$.
Note that
$\Phi_{\hat{\cE}} (\hat{E})$
is the formal completion of some perfect complex on
$X^\prime_R$,
as it belongs to
$D^b(\cX^\prime) \simeq D^b(X^\prime_R)$.
Then,
by the argument in the proof of
\cite[Lemma 3.4]{HMS11},
taking the limit yields an isomorphism
\begin{align*}
f \colon \Phi_{\hat{\cE}} (\hat{E})
\to
G \left( \Phi_\cE (E) \right)
\end{align*}
which completes the proof. 
\end{proof}

\begin{cor} \label{cor:iDeq5} 
The functors
$\Phi_{\hat{\cE}}
\colon
D^b (\cX)
\to
D^b (\cX^\prime)$
and
$\Phi_{\hat{\cE}_K}
\colon
D^b (\Coh (\cX_K))
\to
D^b (\Coh (\cX^\prime_K))$
are equivalences,
where
$\hat{\cE}_K
\in
 D^b (\cX \times_R \cX^\prime) / D^b_0 (\cX \times_R \cX^\prime)$
is the image of
$\hat{\cE}$
by the natural projection. 
\end{cor}

\section{Specialization of derived equivalence}
In this section,
for flat proper families of $\bfk$-varieties over a common base we study the induced derived equivalence from their generic to general fibers.
The key is the ability of Corollary
\pref{cor:CGF}
to lift Fourier--Mukai kernels along the projection.
First,
although it is not strictly necessary for our purpose,
we extend Corollary
\pref{cor:CGF}
to nonaffine base case for flat proper families of $\bfk$-varieties with a smooth fiber.
It suffices to show that,
when restricted to general fibers,
the relative integral functor defined by the lift admits fully faithful left adjoints.
As in the proof of
\cite[Theorem 1.1]{Mor},
we show that
the associated counit morphism is a natural isomorphism.
However,
since in general the generic fiber is not a subscheme of a family,
we have to modify the proof as follows. 
Shrinking the base,
we remove torsion parts with respect to the base from
a fixed strong generator
and
its relevant Hom-sets.
Then we invoke some basic categorical results to see that
the value of the counit morphism on the trimmed strong generator is an isomorphism,
which implies that
the restriction of the counit morphism is a natural isomorphism.

\subsection{Lifts of Fourier--Mukai kernels}
Let
$\pi \colon X \to S$
be a flat proper morphism of $\bfk$-varieties.
Since
$S$
is integral,
the function field
$K = k(S)$
is given by local ring
$\scrO_{X, \xi}$
and
it coincides with the field of fractions
$Q(R)$
for any affine open $\bfk$-subvariety
$U = \Spec R$
\cite[Exercise II3.6]{Har}.
Hence we have the pullback diagram
\begin{align*}
\begin{gathered}
\xymatrix{
X_\xi \ar[r]^-{\bar{\iota}_\xi} \ar_{\pi_\xi}[d] & X \ar^{\pi}[d] \\
\Spec K \ar[r]_-{\iota_\xi} & S
}
\end{gathered}
\end{align*}
where
$\iota_\xi$
is the canonical morphism.

\begin{dfn}
The
\emph{categorical generic fiber}
of
$\pi \colon X \to S$
is the Verdier quotient
\begin{align*} 
D^b(X) / \Ker(\bar{\iota}^*_\xi),
\end{align*}
where
$\Ker(\bar{\iota}^*_\xi)$
is the kernel
\cite[Tag 05RF]{SP}
of the exact functor
$\bar{\iota}^*_\xi \colon D^b(X) \to D^b(X_\xi)$.
\end{dfn}

Recall that
for a smooth separated family
$\pi_R \colon X_R \to \Spec R$
over a nonsingular affine $\bfk$-variety,
by Corollary
\pref{cor:CGF}
there exists a $Q(R)$-linear exact equivalence
\begin{align*} 
D^b(X_R) / D^b_0(X_R)
\simeq
D^b(X_\xi).
\end{align*}
The above definition is an extension of this local description in the following sense.

\begin{thm} \label{thm:gbase}
Assume that
at least one of the closed fibers of
$\pi$
is smooth.
Then there exists a $K$-linear exact equivalence
\begin{align*} 
D^b(X) / \Ker(\bar{\iota}^*_\xi)
\simeq
D^b(X_\xi).
\end{align*}
\end{thm} 
\begin{proof}
Let
$[\bar{\iota}^*_\xi]
\colon
D^b(X) / \Ker(\bar{\iota}^*_\xi)
\to
D^b(X_\xi)$
be the unique functor
which makes the diagram
\begin{align} \label{eq:VQ1}
\begin{gathered}
\xymatrix{
D^b(X) \ar[r]^-{\bar{\iota}^*_\xi} \ar_-{Q}[dr] & D^b(X_\xi) \\
& D^b(X) / \Ker(\bar{\iota}^*_\xi) \ar_-{[\bar{\iota}^*_\xi]}[u]
}
\end{gathered}
\end{align}
commute, 
where
$Q \colon D^b(X) \to D^b(X) / \Ker(\bar{\iota}^*_\xi)$
is the quotient functor.
Take any affine open subset
$U = \Spec R \subset S$.
Let
$\pi_U \colon X_U \to U$
and
$\pi_Z \colon X_Z \to Z$
be the base changes to
$U$
and
its complement
$Z = S \setminus U$
respectively.
We have
\begin{align*}
\coh(X_U) \simeq \coh(X) / \coh_Z(X)
\end{align*}
where the right hand side is the Serre quotient by the Serre subcategory
$\coh_Z(X) \subset \coh(X)$
of sheaves supported on
$X_Z$.
Passing to the derived category,
via
\cite[Theorem 3.2]{Miy}
we obtain $U$-linear exact equivalence
\begin{align*}
D^b(X_U) \simeq D^b(X) / D^b_Z (X)
\end{align*}
where
$D^b_Z (X) \subset D^b(X)$
is the full $S$-linear triangulated subcategory with cohomology supported on
$X_Z$.
Since
$D^b_Z (X)$
is contained in
$\Ker(\bar{\iota}^*_\xi)$,
the commutative diagram
\pref{eq:VQ1}
extends to
\begin{align*}
\begin{gathered}
\xymatrix{
D^b(X) \ar[d]_-{\bar{\iota}^*_U} \ar[r]^-{\bar{\iota}^*_\xi} \ar_{Q}[dr] & D^b(X_\xi) \\
D^b(X_U) \ar[r]_-{[Q]} & D^b(X) / \Ker(\bar{\iota}^*_\xi). \ar_-{[\bar{\iota}^*_\xi]}[u]
}
\end{gathered}
\end{align*}

On the other hand,
the inclusion
$D^b_Z (X) \subset \Ker(\bar{\iota}^*_\xi)$
induces a commutative diagram
\begin{align} \label{eq:VQ2}
\begin{gathered}
\xymatrix{
D^b(X) \ar[d]_-{\bar{\iota}^*_U} \ar[r]^-{\bar{\iota}^*_\xi} & D^b(X_\xi) \\
D^b(X_U) \ar[ur]_-{Q_R} &
}
\end{gathered}
\end{align}
where
$Q_R \colon D^b(X_U) = D^b(X_R) \to D^b(X_\xi) \simeq D^b(X_R) / D^b_0(X_R)$
is the quotient functor.
Note that
shrinking
$U$
if necessary,
we may assume that
$\pi_R$
is smooth in order to apply Corollary
\pref{cor:CGF}.
Indeed,
by
\cite[Theorem I5.3]{Har}
the singular locus of
$X_R$
is a proper closed subset,
whose image under
$\pi_R$
is a proper closed subset of
$\Spec R$.
Note that by assumption there is at least one closed point of
$\Spec R$
which is not contained in the image of the singular locus.
Changing the base to its complement,
we may assume that
$X_R$
is nonsingular.
Then one can apply
\cite[Corollary III10.7]{Har}
to find an open subset of
$\Spec R$
over which the restriction of
$\pi$
becomes smooth.
Since
$D^b_0(X_R)$
is contained in
$\Ker([Q])$,
the commutative diagram
\pref{eq:VQ2}
extends to
\begin{align*}
\begin{gathered}
\xymatrix{
D^b(X) \ar[d]_-{\bar{\iota}^*_U} \ar[r]^-{\bar{\iota}^*_\xi} & D^b(X_\xi) \ar^-{[[Q]]}[d]\\
D^b(X_U) \ar[ur]_-{Q_R} \ar[r]_-{[Q]} & D^b(X) / \Ker(\bar{\iota}^*_\xi)
}
\end{gathered}
\end{align*}
with unique
$[[Q]]$.
Thus we obtain a commutative diagram
\begin{align*}
\begin{gathered}
\xymatrix{
D^b(X) \ar[r]^-{Q} \ar_-{Q}[dr] & D^b(X) / \Ker(\bar{\iota}^*_\xi) \\
& D^b(X) / \Ker(\bar{\iota}^*_\xi). \ar_-{[[Q]] \circ [\bar{\iota}^*_\xi]}[u]
}
\end{gathered}
\end{align*}
By universality of Verdier quotient,
the composition
$[[Q]] \circ [\bar{\iota}^*_\xi]$
is natural isomorphic to the identity functor.
Hence
$[\bar{\iota}^*_\xi]$
is an equivalence.
\end{proof}

\begin{rmk}
The above theorem is a direct consequence of
\cite[Theorem 3.2]{Miy}
and
the $K$-linear equivalence
\begin{align*}
\coh(X) / \Ker(\bar{\iota}^*_\xi)
\simeq
\coh(X_\xi),
\end{align*}
which can be deduced by the similar argument.
Here,
we use the same symbol
$\Ker(\bar{\iota}^*_\xi)$
to denote the kernel
\cite[Tag 02MR]{SP}
of the exact functor
$\bar{\iota}^*_\xi$
of abelian categories.
In particular,
\pref{thm:esurj}
also extends to nonaffine base case for flat proper families of $\bfk$-varieties.
\end{rmk}

\begin{cor}
Assume that
at least one of the closed fibers of
$\pi$
is smooth.
Then any object of 
$D^b(X_\xi)$
can be lifted to that of
$D^b(X)$
along the projection
$Q$.
\end{cor} 

\subsection{Basic categorical results}
\begin{lem} \label{lem:counit1}
Let
$\scrC, \scrD$
be small categories
and
$F \colon \scrC \to \scrD, G \colon \scrD \to \scrC$
functors with
$F \dashv G$.
If there exists an object
$D \in \scrD$
such that
the canonical maps
\begin{align*}
\Hom(D, D) \to \Hom(G(D), G(D)), \
\Hom(D, FG(D) \to \Hom(G(D), GFG(D))
\end{align*}
are bijective,
then the counit morphism
$\epsilon \colon FG \Rightarrow 1_\scrD$
induces an isomorphism
$\epsilon_D \colon FG(D) \to D$.
\end{lem}
\begin{proof}
Let
$\alpha_D, \alpha_{FG(D)}$
be the compositions
\begin{align*}
\begin{gathered}
\Hom(D, D) \to \Hom(G(D), G(D)) \to \Hom(FG(D), D), \\
\Hom(D, FG(D) \to \Hom(G(D), GFG(D)) \to \Hom(FG(D), FG(D))
\end{gathered}
\end{align*}
of the canonical maps.
By assumption
and
definition of adjoint functors
$\alpha_D, \alpha_{FG(D)}$
are bijective.
We denote
$\epsilon_D = \alpha_D(1_D)$
by
$f$
and
$\alpha^{-1}_{FG(D)}(1_{FG(D)})$
by
$g$.
Consider the diagrams
\begin{align*}
\begin{gathered}
\xymatrix{
\Hom(D, D) \ar[r]^-{\alpha_D} \ar_{g \circ}[d] & \Hom(FG(D), D) \ar^{g \circ}[d] & \Hom(D, FG(D)) \ar[r]^-{\alpha_{FG(D)}}\ar_{f \circ}[d] & \Hom(FG(D), FGD) \ar^{f \circ}[d] \\
\Hom(D, FG(D)) \ar[r]_-{\alpha_{FG(D)}} & \Hom(FG(D), FGD), & \Hom(D, D) \ar[r]_-{\alpha_D} & \Hom(FG(D), D),
}
\end{gathered}
\end{align*}
which are commutative by
definition of functors
and
naturality of adjoints. 
As expressions of the images of
$1_D, g$
we respectively obtain
\begin{align*}
\begin{gathered}
1_{FG(D)}
=
\alpha_{FG(D)}(\alpha^{-1}_{FG(D)}(1_{FG(D)}))
=
\alpha_{FG(D)}(g)
=
g(\alpha_{D}(1_D))
=
gf, \\
\alpha_D(fg)
=
f(\alpha_{FG(D)}(g))
=
f(\alpha_{FG(D)}(\alpha^{-1}_{FG(D)}(1_{FG(D)})))
=
f.
\end{gathered}
\end{align*}
From the second line it follows
$fg = \alpha^{-1}_D(f) = 1_D$.
Hence
$f$
is an isomorphism.
\end{proof}

\begin{rmk}
The above proof is just an adaptation of the proof of the fact that
$G$
is fully faithful
if and only if
$\epsilon$
is a natural isomorphism.
\end{rmk}

Similarly,
one can prove the dual statement.

\begin{lem} \label{lem:unit1}
Let
$\scrC, \scrD$
be small categories
and
$F \colon \scrC \to \scrD, G \colon \scrD \to \scrC$
functors with
$F \dashv G$.
If there exists an object
$C \in \scrC$
such that
the canonical maps
\begin{align*}
\Hom(C, C) \to \Hom(F(D), F(D)), \
\Hom(GF(C), C) \to \Hom(FGF(C), F(C))
\end{align*}
are bijective,
then the unit morphism
$\eta \colon 1_\scrC \Rightarrow GF$
induces an isomorphism
$\eta_C \colon C \to GF(C)$.
\end{lem}

\begin{rmk}
Again,
the proof is just an adaptation of the proof of the fact that
$F$
is fully faithful
if and only if
$\eta$
is a natural isomorphism.
\end{rmk}

\subsection{Removal of torsion supports}
Let
$\pi_R \colon X \to \Spec R, \pi^\prime_R \colon X^\prime \to \Spec R$
be a smooth proper morphisms to a nonsingular affine $\bfk$-variety.
Assume that
their generic fibers
$X_K, X^\prime_K$
are derived-equivalent.
Let
$\Phi_{\cE_K} \colon D^b(X_K) \to D^b(X^\prime_K)$
be a Fourier--Mukai functor giving the equivalence
with kernel
$\cE_K \in D^b(X_K \times X^\prime_K)$.
Fix a lift
$\cE \in D^b(X \times_R X^\prime)$
of
$\cE_K$
along the projection
\begin{align*}
D^b(X \times_R X^\prime)
\to
D^b(X \times_R X^\prime) / D^b_0 (X \times_R X^\prime)
\simeq
D^b(X_K \times X^\prime_K)
\end{align*}
and
a strong generator
$E$
of
$D^b(X)$.

\begin{lem} \label{lem:counit2}
There exists a nonempty affine open subset
$U \subset \Spec R$
over which the restriction
\begin{align*}
\Phi_U
=
\Phi_{\cE_U}
\colon
D^b(X_U)
\to
D^b(X^\prime_U)
\end{align*}
induces bijections
\begin{align} \label{eq:counit}
\begin{gathered}
\Hom(E_U, E_U)
\to
\Hom(\Phi_U(E_U), \Phi_U(E_U)), \\
\Hom(E_U, \Phi^L_U \Phi_U(E_U))
\to
\Hom(\Phi_U(E_U), \Phi_U \Phi^L_U \Phi_U(E_U))
\end{gathered}
\end{align}
where
$\Phi^L_U \colon D^b(X^\prime_U) \to D^b(X_U)$
is the left adjoint to
$\Phi_U$.
\end{lem}
\begin{proof}
By
\cite[Theorem 2.1.2, Lemma 3.4.1]{BV}
the base change
$E_K$
of
$E$
along the canonical inclusion
$R \hookrightarrow K$
is a strong generator of
$D^b(X_K)$,
which cannot be trivial.
Consider the support of the $R$-torsion part of
$E$,
i.e.,
the union
$\bigcup_{i} \supp \cH^i(E)_{tors}$
of the supports of the $R$-torsion parts
$\cH^i(E)_{tors}$
of
$\cH^i(E)$.
Each
$\cH^i(E)_{tors}$
is a coherent sheaf on
$X$,
as every submodule of a finitely presented module over a noetherian ring is finitely presented. 
Since the union is finite,
$\bigcup_{i} \supp \cH^i(E)_{tors}$
is a closed subset of
$X$.
Its complement must contain the generic point of
$X$,
otherwise
$E_K$
is trivial.
Let
$U \subset \Spec R$
be the image of the complement under
$\pi_R$,
which is a nonempty open subset.
By construction over
$U$
the restriction
$E_U = E |_{\pi^{-1}_U(U)}$
is $\scrO_U(U)$-torsion free. 
Since
$\Hom(E_U, E_U)$
is coherent as an $\scrO_U(U)$-module by
\pref{lem:coh}
below,
shrinking
$U$
if necessary,
we may assume that
it is $\scrO_U(U)$-torsion free.
The tensor product
$\Phi_U(E_U) \otimes_{\scrO_U(U)} K$
cannot be trivial,
otherwise it does not map to an object quasi-isomorphic to
$E_K$
under
$\Phi^{-1}_K$.
Hence by the same argument one finds an affine open subset
$V \subset \Spec R$
over which
the restriction
$\Phi_V (E_V)
\cong
\Phi_\cE (E) |_{\pi^{\prime -1}_V(V)}$
and
$\Hom(\Phi_V(E_V), \Phi_V(E_V))$
are $\scrO_V(V)$-torsion free.
The intersection
$U \cap V \subset \Spec R$
is a nonempty open subset,
as
$\Spec R$
is integral.
Now,
we replace
$U \cap V$
with
$U$.

Consider a map
\begin{align} \label{eq:S}
\Hom(E_U, E_U)
\xrightarrow{\Phi_U}
\Hom(\Phi_U(E_U), \Phi_U(E_U))
\end{align}
of $\scrO_U(U)$-modules.
By assumption
and
\pref{lem:VKlin}
tensoring
$K$
yields an exact sequence
\begin{align} \label{eq:ES}
0
\to
\Hom(E_U, E_U) \otimes_{\scrO_U(U)} K
\xrightarrow{\Phi_U \otimes_{\scrO_U(U)} K}
\Hom(\Phi_U(E_U), \Phi_U(E_U)) \otimes_{\scrO_U(U)} K
\to
0.
\end{align}
Hence the map
$\Phi_U$
in
\pref{eq:S}
is injective,
as
$\Hom(E_U, E_U)$
is $\scrO_U(U)$-torsion free.
The associated sheaf with the cokernel
\begin{align*}
\Hom(\Phi_U(E_U), \Phi_U(E_U)) / \Hom(E_U, E_U)
\end{align*}
might have nontrivial $\scrO_U(U)$-torsion part.
However,
since by
\pref{lem:coh}
below it is coherent,
one finds an affine open subset
$W \subset U$
to which the restriction of the associated sheaf is $\scrO_W(W)$-torsion free.
Now,
we replace
$W$
with
$U$.
Then the exactness of
\pref{eq:ES}
implies that of
\pref{eq:S}.
Hence the map
$\Phi_U$
in
\pref{eq:S}
is bijective.
Shrinking
$U$
if necessary,
by the same argument we conclude that
\begin{align*}
\Hom(E_U, \Phi^L_U \Phi_U(E_U))
\xrightarrow{\Phi_U}
\Hom(\Phi_U(E_U), \Phi_U \Phi^L_U \Phi_U(E_U))
\end{align*}
is also bijective.
\end{proof}

\begin{rmk}
Note that the base change
\begin{align} \label{eq:BC}
(\Hom(\Phi_U(E_U), \Phi_U(E_U)) / \Hom(E_U, E_U)) \otimes_{\scrO_U(U)} \scrO_W(W)
\end{align}
is isomorphic to the cokernel of the sequence
\begin{align*}
0
\to
\Hom(E_W, E_W)
\xrightarrow{\Phi_W}
\Hom(\Phi_W(E_W), \Phi_W(E_W))
\to
0.
\end{align*}
Indeed,
\pref{eq:BC}
is isomorphic to
\begin{align*}
\Hom(\Phi_U(E_U), \Phi_U(E_U)) \otimes_{\scrO_U(U)} \scrO_W(W)  / \Hom(E_U, E_U) \otimes_{\scrO_U(U)} \scrO_W(W) 
\end{align*}
as the pullback by an open immersion is exact.
Consider the pullback diagrams
\begin{align*}
\begin{gathered}
\xymatrix{
X_W \ar@{^{(}->}[r]^-{\bar{\iota}} \ar_{\pi_W}[d] & X_U \ar^{\pi_U}[d] & X^\prime_W \ar@{^{(}->}[r]^-{\bar{\iota}^\prime}\ar_{\pi^\prime_W}[d] & X^\prime_U \ar^{\pi^\prime_U}[d] \\
W \ar@{^{(}->}[r]_-{\iota} & U, & W \ar@{^{(}->}[r]_-{\iota} & U.
}
\end{gathered}
\end{align*}
By the derived flat base change we have
\begin{align*}
\begin{gathered}
\iota^* R \Hom^\bullet(E_U, E_U)
\cong
\iota^* R \pi_{U *} R \underline{\Hom}^\bullet(E_U, E_U)
\cong
R \Hom^\bullet(E_W, E_W), \\
\iota^{\prime *} R \Hom^\bullet(\Phi_U(E_U), \Phi_U(E_U))
\cong
\iota^{\prime *} R \pi^\prime_{U *} R \underline{\Hom}^\bullet(\Phi_U(E_U), \Phi_U(E_U))
\cong
R \Hom^\bullet(\Phi_W(E_W), \Phi_W(E_W)). 
\end{gathered}
\end{align*}
Passing to the $0$-th cohomology of complexes,
we obtain
\begin{align*}
\begin{gathered}
\Hom(E_U, E_U) \otimes_{\scrO_U(U)} \scrO_W(W)
\cong
\Hom(E_W, E_W), \\
\Hom(\Phi_U(E_U), \Phi_U(E_U)) \otimes_{\scrO_U(U)} \scrO_W(W)
\cong
\Hom(\Phi_W(E_W), \Phi_W(E_W))
\end{gathered}
\end{align*}
as
$\iota^*, \iota^{\prime *}$
are exact.
\end{rmk}

\begin{lem} \label{lem:coh}
For any object
$E, F \in D^b(X)$
the $R$-module
$\Hom(E, F) = \Ext^0_X(E, F)$
is coherent.
\end{lem}
\begin{proof}
This is a standard fact.
\end{proof}

Similarly,
one can prove the dual statement.

\begin{lem} \label{lem:unit2}
There exists a nonempty affine open subset
$U \subset \Spec R$
over which the restriction
\begin{align*}
\Phi^L_U
=
\Phi^L_{\cE_U}
\colon
D^b(X^\prime_U)
\to
D^b(X_U)
\end{align*}
induces bijections
\begin{align} \label{eq:unit}
\begin{gathered}
\Hom(E^\prime_U, E^\prime_U)
\to
\Hom(\Phi^L_U(E^\prime_U), \Phi^L_U(E^\prime_U)), \\
\Hom(\Phi_U \Phi^L_U (E^\prime_U), E^\prime_U)
\to
\Hom(\Phi^L_U \Phi_U \Phi^L_U (E^\prime_U), \Phi^L_U (E^\prime_U)).
\end{gathered}
\end{align}
\end{lem}

\subsection{Specialization}
\begin{thm} \label{thm:main}
There exists a nonempty affine open subset
$U \subset \Spec R$
over which the restriction
\begin{align*}
\Phi_U
=
\Phi_\cE |_{\pr^{-1}_1 \circ \pi^{-1}_U(U)}
\colon
D^b(X_U)
\to
D^b(X^\prime_U)
\end{align*}
becomes an $\scrO_U(U)$-linear exact equivalence.
In particular,
over
$U$
any pair of closed fibers are derived-equivalent.
\end{thm}
\begin{proof}
The proof is an adaptation of the argument in the proof of
\cite[Theorem 1.1]{Mor}.
Fix a strong generator
$E$
of
$D^b (X)$.
The counit morphism
$\epsilon \colon \Phi^L_\cE \circ \Phi_\cE \to \id_{D^b (X)}$
gives a distinguished triangle
\begin{align} \label{eq:DT}
\Phi^L_\cE \circ \Phi_\cE (E) \xrightarrow{\epsilon_E} E \to F \coloneqq \Cone \left( \epsilon \left( E \right) \right).
\end{align}
Over any open subset
$U \subset \Spec R$,
\pref{eq:DT}
restricts to a distinguished triangle
\begin{align*}
\Phi^L_U \circ \Phi_U (E_U) \xrightarrow{\epsilon_{E_U}} E_U \to F_U.
\end{align*}
Note that the restriction of the counit morphism is the counit morphism.
Choose
$U$
so that
we have the bijections
\pref{eq:counit}
from
\pref{lem:counit2}.
Then by
\pref{lem:counit1}
the counit morphism
$\epsilon_{E_U}$
on
$E_U$
is an isomorphism.
Since
$E_U$
is a strong generator of
$D^b(X_U)$
by
\cite[Theorem 2.1.2, Lemma 3.4.1]{BV},
this implies that
$\Phi_U$
is fully faithful.
Recall that
a triangulated category is strongly finitely generated
if there exist an object
$E_U$
and
a nonnegative integer
$k$
such that
every object can be obtained from
$E_U$
by taking
isomorphisms,
finite direct sums,
direct summands,
shifts,
and
not more than
$k$
times cones.
Now,
we may assume that
$E_U$
has no nontrivial direct summands,
as
$\Phi_U$
and
$\Phi^L_U$
commute with direct sums on
$D^b(X_U)$
by
\cite[Corollary 3.3.4]{BV}.
Since
$\epsilon_{E_U}$
is an isomorphism,
one inductively sees that 
over
$U$
the counit morphism on any object is an isomorphism.

Fix a strong generator
$E^\prime$
of
$D^b (X^\prime)$.
The unit morphism
$\eta \colon \id_{D^b (X^\prime)} \to \Phi_\cE \circ \Phi^L_\cE$
gives a distinguished triangle
\begin{align} \label{eq:DT2}
E^\prime \xrightarrow{\eta_{E^\prime}} \Phi_\cE \circ \Phi^L_\cE(E^\prime) \to F^\prime \coloneqq \Cone (\eta_{E^\prime}).
\end{align}
Over any open subset
$U \subset \Spec R$,
\pref{eq:DT2}
restricts to a distinguished triangle
\begin{align*}
E^\prime_U \xrightarrow{\eta_{E^\prime_U}} \Phi_U \circ \Phi_U (E^\prime_U) \to F^\prime_U.
\end{align*}
Note that the restriction of the unit morphism is the unit morphism.
Choose
$U$
so that
we have the bijections
\pref{eq:unit}
from
\pref{lem:unit2}.
Then by
\pref{lem:unit1}
the unit morphism
$\eta_{E^\prime_U}$
on
$E^\prime_U$
is an isomorphism.
Since
$E^\prime_U$
is a strong generator of
$D^b(X^\prime)$
by
\cite[Theorem 2.1.2, Lemma 3.4.1]{BV},
this implies that
$\Phi^L_U$
is fully faithful.
Shrinking
$U$
if necessary,
we may assume that
over
$U$
both
$\Phi_U$
and
$\Phi^L_U$
are fully faithful.
Then
$\Phi_U$
is an equivalence,
as a fully faithful functor
which admits a fully faithful left adjoint
is an equivalence.
\end{proof}

\begin{rmk}
If
$\Phi_\cE$
induces the derived equivalence of a single pair of closed fibers,
then there exists a nonempty Zariski open subset
$U \subset \Spec R$
such that the base changes
$X_U, X^\prime_U$
are derived-equivalent.
This follows from the proof of
\cite[Theorem 1.1]{Mor},
which exploits the fact that
the restriction of the counit morphism
$\epsilon_E
\colon
\Phi^L_\cE \circ \Phi_\cE(E) \to E$
to any closed fiber is the counit morphism for each object
$E \in D^b (X)$.
However,
it does not work for the generic fiber.
In general,
the generic fiber is not a subscheme of
$X_R$,
while any closed fiber can naturally be regarded as a subscheme of
$X_R$
via the reduced induced structure on the image of the closed immersion.
\end{rmk}

\begin{cor} \label{cor:1.4}
Let
$\pi \colon X \to S, \pi^\prime \colon X^\prime \to S$
be flat proper morphisms of $\bfk$-varieties.
Assume that
their generic fibers
$X_K, X^\prime_K$
are derived-equivalent.
Assume further that
at least one of the closed fibers of each
$\pi$
and
$\pi^\prime$
are smooth.
Then there exists a nonempty open subset
$U \subset S$
to which the base changes
$X_U, X^\prime_U$
become derived-equivalent.
In particular,
over
$U$
any pair of closed fibers are derived-equivalent.
\end{cor}
\begin{proof}
By the same argument as in the proof of
\pref{thm:gbase}
we may assume that
$\pi, \pi^\prime$
are smooth.
Let
$\Phi_{\cE_K} \colon D^b(X_K) \to D^b(X^\prime_K)$
be a Fourier--Mukai functor giving the equivalence
with kernel
$\cE_K \in D^b(X_K \times X^\prime_K)$.
Fix a lift
$\cE \in D^b(X \times_S X^\prime)$
of
$\cE_K$
along the projection
\begin{align*}
D^b(X \times_S X^\prime)
\to
D^b(X \times_S X^\prime) / D^b_0 (X \times_S X^\prime)
\simeq
D^b(X_K \times X^\prime_K).
\end{align*}
Take an affine open cover
$\bigcup^N_{i = 1} \Spec R_i$
of
$S$.
One can apply
\pref{thm:main}
to find open subsets
$U_i \subset \Spec R_i$
over which the restrictions
\begin{align*}
\Phi_{U_i}
=
\Phi_{\cE_{U_i}}
\colon
D^b(X_{U_i})
\to
D^b(X^\prime_{U_i})
\end{align*}
becomes $\scrO_{U_i}(U_i)$-linear exact equivalences.
Let
$V = \bigcup^N_{i=1} U_i$
be their union,
which is an open $\bfk$-subvariety of
$S$.
Consider the restriction
\begin{align*}
\Phi_{V}
=
\Phi_{\cE_V}
\colon
D^b(X_V)
\to
D^b(X^\prime_V)
\end{align*}
over
$V$.
Since its restriction to any pair of closed fibers over
$V$
defines an equivalence,
$\Phi_V$
is an equivalence by
\cite[Proposition 2.15]{HLS}.
\end{proof}

\section{Fourier--Mukai partners over the closure of function fields}
In this section,
passing through the deformation theory,
we provide new examples of Fourier--Mukai partners.
The above results play a role in deducing the derived equivalences.
Let
$X_0, X_0^\prime$
be derived-equivalent Calabi--Yau manifolds
of dimension more than two.
There exists a nonsingular affine $\bfk$-variety
$\Spec S$
such that
smooth projective versal deformations
$X_S, X^\prime_S$
over
$S$
are derived-equivalent
\cite[Theorem 1.1]{Mor}.
Also effectivizations
$X_R, X^\prime_R$
are shown to be derived-equivalent.
These equivalences are given by deformed Fourier--Mukai kernels.
From our results in Section
$4$
it immediately follows that
their generic fibers are derived-equivalent.
One can check that
they are Calabi--Yau manifolds.
If
$X_0, X^\prime_0$
satisfy either
$\NS_{tor} X_0 \neq \NS_{tor} X^\prime_0$,
or
$\rho (X_0) = \rho (X^\prime_0) = 1$
and
$\deg (X_0) \neq \deg (X^\prime_0)$,
then the generic fibers are nonbirational
as well as
$X_0, X^\prime_0$.
Several pairs are known to satisfy either of the two conditions.
Thus we obtain new examples of Fourier--Mukai partners over the function fields
$K, Q$.
One may also pass to the base changes over a smooth connected $\bfk$-curve containing
$X_0, X^\prime_0$
to apply
\cite[Theorem 1.1]{KT}.
Similar results hold for the geometric generic fibers.
If
$\bfk$
is a universal domain,
then
$\bar{Q}$
is isomorphic to
$\bfk$
but
the Fourier--Mukai partner
$X_{\bar{Q}}, X^\prime_{\bar{Q}}$
are different as a $\bfk$-variety from known examples.
We demonstrate this subtle difference
when
$X_0, X^\prime_0$
are the famous Pfaffian--Grassmannian pair.

\subsection{Deformations of Calabi--Yau manifolds}
Let
$X_0$
be a Calabi--Yau manifold with dimension more than two.
The deformation functor
\begin{align*}
\Def_{X_0} \colon \Art \to \Set
\end{align*}
of
$X_0$
has a universal formal family
$(R, \xi)$,
where
$R$
is a formal power series ring with
$\dim_{\bfk} \text{H}^1 (X_0, \scrT_{X_0})$
valuables,
and
$\xi$ belongs to the limit 
\begin{align*}
\widehat{\Def}_{X_0} (R)
=
\displaystyle \lim_{\longleftarrow} \Def_{X_0} (R / \frakm^n_R)
\end{align*}
of the inverse system
\begin{align*}
\cdots \to \Def_{X_0} (R/\frakm_R^{n+2}) \to \Def_{X_0} (R/\frakm_R^{n+1}) \to \Def_{X_0} (R/\frakm_R^n) \to \cdots
\end{align*}
induced by the natural quotient maps
$R/\frakm^{n+1}_R \to R/\frakm^n_R$.
Here,
$\frakm_R \subset R$
is the maximal ideal.

The formal family
$\xi$
corresponds to a natural transformation 
\begin{align*}
h_R = \Hom_{\operatorname{\bfk-alg}} (R, -) \to \Def_{X_0},
\end{align*}
which sends each homomorphism
$f \in h_R (A)$
factorizing through
\begin{align*}
R \to R / \frakm^{n+1}_R \xrightarrow{g} A
\end{align*}
to
$\Def_{X_0} (g) (\xi_n)$.
Let
$X_n$
be the schemes defining
$\xi_n$.
There is a noetherian formal scheme
$\cX$
over
$R$
such that
$X_n \cong \cX \times_R R / \frakm^{n+1}_R$
for each
$n$.
Thus for a deformation
$(X_A, i_A)$
the scheme
$X_A$
can be obtained as the pullback of
$\cX$
along some morphism of noetherian formal schemes
$\Spec A \to \Spf R$.

Now,
we briefly recall how to algebrize
$\cX$.
By
\cite[Theorem III5.4.5]{GD}
there exists a scheme
$X_R$
flat projective over
$R$
whose formal completion along the closed fiber
$X_0$
is isomorphic to
$\cX$.
Moreover,
$X_R$
is smooth over
$R$
of relative dimension
$\dim X_0$
\cite[Lemma 2.4]{Mor}.
We call 
$X_R$
an effectivization
of
$\cX$.
Consider the extended functor 
\begin{align*}
\Def_{X_0}
\colon
\operatorname{\textbf{k-alg}^\text{aug}}
\to
\Set
\end{align*}
from the category of augmented noetherian $\bfk$-algebras.
Let
$T = \bfk [ t_1, \ldots, t_d ]$
and
$t \in \Spec T$
be the closed point corresponding to maximal ideal
$(t_1, \ldots, t_d)$.
There is a filtered inductive system
$\{ R_i \}_{i \in I}$
of finitely generated $T$-subalgebras of
$R$
whose colimit is
$R$.
Since
$\Def_{X_0}$
is locally of finite presentation,
$[ X_R, i_R ]$
is the image of some element
$\zeta_i
\in
\Def_{X_0} \left( \left( R_i, \frakm_{R_i} \right) \right)$
by the canonical map
$\Def_{X_0} \left( \left( R_i, \frakm_{R_i} \right) \right)
\to
\Def_{X_0} (R)$.
By
\cite[Corollary 2.1]{Art69a}
there exists an \'etale neighborhood
$\Spec S$
of
$t$
in
$\Spec T$
with
first order approximation
$\varphi \colon R_i \to S$
of
$R_i \hookrightarrow R$.
Let
$[ X_S, i_S ]$
be the image of
$\zeta_i$
by the map
$\Def_{X_0} (\varphi)$.
The formal completion of
$X_S$
along the closed fiber
$X_0$
is isomorphic to
$\cX$. 

The triple
$(\Spec S, s_0, X_S)$,
or sometimes
$X_S$,
is called a versal deformation of
$X_0$.
By construction,
$\Spec S$
is an algebraic $\bfk$-variety with a distinguished closed point
$s_0$
mapping to
$t$,
and
$X_S$
is flat of finite type over
$S$
whose closed fiber
over
$s_0$
is
$X_0$.
In our setting,
one can find a versal deformation
$X_S$
smooth projective over a nonsingular affine $\bfk$-variety
$\Spec S$
of relative dimension
$\dim X_0$
\cite[Lemma 2.3]{Mor}. 
Moreover,
given another Calabi--Yau manifold
$X^\prime_0$
derived-equivalent to
$X_0$,
one can find a smooth projective versal deformation
$X^\prime_S$
over the same base.
See
\cite[Section 3]{Mor}
for the detail.
The point is the facts that
by
\cite{CW}
the derived equivalence induces an isomorphism
\begin{align*}
\text{H}^1 (X_0, \scrT_{X_0})
=
\text{HH}^2(X_0)
\cong
\text{HH}^2(X^\prime_0)
=
\text{H}^1 (X^\prime_0, \scrT_{X^\prime_0})
\end{align*}
and
by Bogomolov-Tian-Todorov theorem deformations of Calabi--Yau manifolds are unobstructed. 
The construction passes through effectivizations.
Namely,
there are effectivizations
$X_R, X^\prime_R$
of
$X_0, X^\prime_0$
over the same regular affine $\bfk$-scheme
$\Spec R$.

\subsection{Calabi--Yau geometric generic fibers}
First,
we consider the effectivization
$X_R$
of
$\cX$
and
its geometric generic fiber.
We have the pullback diagram
\begin{align*}
\begin{gathered}
\xymatrix{
X_{\bar{K}} \ar[r]^-{\bar{i}} \ar_{\pi_{\bar{K}}}[d] & X_K \ar[r]^-{i} \ar_{\pi_K}[d] & X_R \ar^{\pi_R}[d] \\
\Spec \bar{K} \ar[r]_-{\bar{j}} & \Spec K \ar[r]_-{j} & \Spec R.
}
\end{gathered}
\end{align*}

\begin{lem}
The geometric generic fiber
$X_{\bar{K}}$
is a Calabi--Yau manifold over
$\bar{K}$.
\end{lem}
\begin{proof}
Smoothness and projectivity
follow from their being stable under base change.
One can apply
\cite[Proposition IV15.5.7]{GD66}
to see that
$X_{\bar{K}}$
is connected.
Then
$X_{\bar{K}}$
must be irreducible,
as it is regular.
By
\cite[Theorem III12.8]{Har} 
the function
$h^0 \colon \Spec R  \to \bZ$
defined as
\begin{align*}
h^0(r)
=
\dim_\bfk \text{H}^0 (X_r, \wedge^{\dim X_0 -1} \scrT_{\pi_R} \otimes_R k(r))
\end{align*}
for
$r \in \Spec R$
is upper semicontinuous,
where
$\scrT_{\pi_R}$
is the relative tangent sheaf.
It follows that
there is an open neighborhood
$U$
of the closed point
to which the restriction vanishes.
Since
$R$
is a domain,
$U$ contains the generic point.
By flat base change we obtain
\begin{align*}
\text{H}^0 (X_{\bar{K}}, \wedge^{\dim X_0 -1} \scrT_{X_{\bar{K}}})
\cong
\text{H}^0 (X_K, \wedge^{\dim X_0 -1} \scrT_{X_K}) \otimes_K \bar{K} 
=
0.
\end{align*}
Similarly,
one can show the vanishing of all the other relevant cohomology.

It remains to show the triviality of the canonical bundle.
Consider the formal completion
$\hat{\omega}_{X_R / R}$
of the relative canonical sheaf
on
$X_R$
along the closed fiber
$X_0$.
It is given by the limit of inverse system
$\{ \omega_{X_{R_n} / R_n} \}_{n \in \bN}$
with
$R_n = R / \frakm^{n+1}_R$.
Here,
the inverse system consists of the sequence of deformations of
$\omega_{X_0}$
along order by order square zero extensions. 
Since we have
$\omega_{X_0} \cong \scrO_{X_0}$,
the inverse system
$\{ \scrO_{X_{R_n}} \}_{n \in \bN}$
also consists of the sequence of deformations of
$\omega_{X_0}$.
On the other hand,
by
\cite[Theorem 3.1.1(2)]{Lie}
the set of deformations of   
$\omega_{X_0}$
as a perfect complex to
$X_1$
is a pseudo-torsor under
$\Ext^1_{X_0}(\omega_{X_0}, \omega_{X_0})^{\oplus l_1}$,
where
$l_1$
is the dimension of $\bfk$-vector space
$\frakm_R / \frakm^2_R$.
This is trivial by the assumption on
$X_0$
and
there is an isomorphism
$\omega_{X_1 / R_1} \cong \scrO_{X_{R_1}}$
respecting 
$\omega_{X_1 / R_1} \otimes_{R_1} \bfk
\cong
\omega_{X_0}$
and
$\scrO_{X_{R_1}} \otimes_{R_1} \bfk
\cong
\scrO_{X_0}
\cong
\omega_{X_0}$.
Inductively,
one finds isomorphisms
$\omega_{X_n / R_n} \cong \scrO_{X_{R_n}}$
respecting
$\omega_{X_n / R_n} \otimes_{R_n} R_{n-1}
\cong
\omega_{X_{n-1}}$
and
$\scrO_{X_{R_n}} \otimes_{R_n} R_{n-1}
\cong
\scrO_{X_{n-1}}
\cong
\omega_{X_{n-1}}$.
By universality of limit,
we obtain
$\hat{\omega}_{X_R / R} \cong \hat{\scrO}_{X_R}$,
which in turn induces
$\omega_{X_R / R} \cong \scrO_{X_R}$
via the equivalence
\pref{eq:Alg}.
\end{proof}

Next,
we consider the versal deformation
$X_S$
of
$X_0$
and
its geometric generic fiber.
We have the pullback diagram
\begin{align*}
\begin{gathered}
\xymatrix{
X_{\bar{Q}} \ar[r]^-{\bar{u}} \ar_{\pi_{\bar{Q}}}[d] & X_{\bar{Q}} \ar[r]^-{u} \ar_{\pi_Q}[d] & X_S \ar^{\pi_S}[d] \\
\Spec \bar{Q} \ar[r]_-{\bar{v}} & \Spec Q \ar[r]_-{v} & \Spec S.
}
\end{gathered}
\end{align*}

\begin{lem}
The geometric generic fiber
$X_{\bar{Q}}$
is a Calabi--Yau manifold over
$\bar{Q}$.
\end{lem}
\begin{proof}
Nontrivial part is the triviality of the canonical bundle.
Consider the collection
$\{ \omega_{X_{R_i} / R_i} \}_{i \in I}$
of relative canonical sheaves
on
$X_{R_i}$.
It consists of the sequence of deformations of
$\omega_{X_0}$.
Since we have
$\omega_{X_0} \cong \scrO_{X_0}$,
the collection
$\{ \scrO_{X_{R_n}} \}_{n \in \bN}$
of structure sheaves also consists of the sequence of deformations of
$\omega_{X_0}$.
We have
$\omega_{X_R / R}
\cong
\omega_{X_{R_i} / R_i} \otimes_{R_i} R$
by
\cite[Proposition II8.10]{Har}
and
the construction of
$X_S$.
One can apply
\cite[Proposition 2.2.1]{Lie}
to find an isomorphism
$\omega_{X_{R_i} / R_i} \cong \scrO_{X_{R_i}}$
for sufficiently large
$i$
with respect to the partial order of
$I$.  
We obtain
$\omega_{X_S / S} \cong \scrO_{X_S}$,
again by
\cite[Proposition II8.10]{Har}
and
the construction of
$X_S$.
\end{proof}

\begin{rmk}
Assuming
$\bfk = \bC$,
one can show the previous lemma
without using the deformation theory of perfect complexes as follows. 
The fact that
deformations of Calabi--Yau manifolds are Calabi--Yau is well-known.
Let
$X_\bfk$
be a very general fiber of
$X_S$.
Choose a countable subfield
$L \subset \bfk$
over which
$\pi_S$
is defined,
i.e.,
we are given a morphism of $L$-varieties
$\pi_L \colon X_L \to \Spec S_L$
such that
the base change
$\pi_L \times_L \bfk$
is isomorphic to
$\pi_S$.
From the proof of
\cite[Lemma 2.1]{Via}
there exists an isomorphism
$\bfk \to \bar{Q}$
fixing
$L$
such that
the base change
$X_\bfk \times_\bfk {\bar{Q}}$
is isomorphic to
$X_{\bar{Q}}$.
Note that
$X_\bfk$
and
$X_{\bar{Q}}$
are isomorphic
as a scheme
but not
as a variety,
since the induced isomorphism of the schemes does not lie over a fixed base.
\end{rmk}

\subsection{The derived equivalence}
Suppose that
$X_0$
is derived-equivalent to another Calabi--Yau manifold
$X^\prime_0$.
Recall that there are effectivizations
$X_R, X^\prime_R$
over the same regular affine $\bfk$-scheme
$\Spec R$.
We have the pullback diagram 
\begin{align*}
\begin{gathered}
\xymatrix{
X_{\bar{K}} \times X^\prime_{\bar{K}} \ar[r]^-{\bar{i}^{\prime \prime}} \ar_{\pi_{\bar{K}} \times \pi^\prime_{\bar{K}}}[d] & X_K \times X^\prime_K \ar[r]^-{i^{\prime \prime}} \ar_{\pi_K \times \pi^\prime_K}[d] & X_R \times_R X^\prime_R \ar^{\pi_R \times_R \pi^\prime_R}[d] \\
\Spec \bar{K} \ar[r]_-{} & \Spec K \ar[r]_-{} & \Spec R.
}
\end{gathered}
\end{align*}
Let
$\cE_0$
be a Fourier--Mukai kernel. 
By
\cite[Proposition 3.3, Corollary 4.2]{Mor}
one can deform
$\cE_0$
to a Fourier--Mukai kernel
$\cE$
on
$X_R \times_R X^\prime_R$. 
Applying
Proposition
\pref{prop:iDeq2}
and
Corollary
\pref{cor:iDeq3},
we obtain the derived equivalence of the geometric generic fibers
\begin{align*}
\Phi_{(i^{\prime \prime} \circ {\bar{i}}^{\prime \prime})^* \cE}
\colon
D^b(X_{\bar{K}})
\to
D^b(X^\prime_{\bar{K}}).
\end{align*}

Recall that
there are smooth projective versal deformations
$X_S, X^\prime_S$
of
$X_0, X^\prime_0$
over the same nonsingular affine $\bfk$-variety
$\Spec S$.
We have the pullback diagram 
\begin{align*}
\begin{gathered}
\xymatrix{
X_{\bar{Q}} \times X^\prime_{\bar{Q}} \ar[r]^-{\bar{j}^{\prime \prime}} \ar_{\pi_{\bar{Q}} \times \pi^\prime_{\bar{Q}}}[d] & X_Q \times X^\prime_Q \ar[r]^-{j^{\prime \prime}} \ar_{\pi_Q \times \pi^\prime_Q}[d] & X_S \times_S X^\prime_S \ar^{\pi_S \times_S \pi^\prime_S}[d] \\
\Spec \bar{Q} \ar[r]_-{} & \Spec Q \ar[r]_-{} & \Spec S.
}
\end{gathered}
\end{align*}
By
\cite[Proposition 3.3]{Mor}
one can deform
$\cE_0$
to a perfect complex
$\cE_S$
on
$X_S \times_S X^\prime_S$.
After possible shrinking of
$\Spec S$,
the relative integral functor
$\Phi_{\cE_S}$
is an equivalence
\cite[Theorem 1.1]{Mor}.
Applying
Proposition
\pref{prop:iDeq2}
and
Corollalry
\pref{cor:iDeq3},
we obtain the derived equivalence of the geometric generic fibers
\begin{align*}
\Phi_{(j^{\prime \prime} \circ {\bar{j}}^{\prime \prime})^* \cE}
\colon
D^b(X_{\bar{Q}})
\to
D^b(X^\prime_{\bar{Q}}).
\end{align*}

\subsection{Nonbirationality}
Let
$X$
be a smooth proper variety over an algebraically closed field.
Recall that
the Néron--Severi group
$\NS X$
is the quotient of the Picard group
$\Pic X$
by the subgroup
$\Pic^0 X$
of isomorphism classes of line bundles
which are algebraically equivalent to
$0$.
The group
$\NS X$
is a finitely generated with rank
$\rho (X)$
called the Picard number.
We denote by
$\NS_{tor} X$
the subgroup of torsion elements,
which is known to be a birational invariant.

\begin{lem}
If
$\NS_{tor} X_0, \NS_{tor} X^\prime_0$
are nonisomorphic,
then
$X_{\bar{K}}, X^\prime_{\bar{K}}$
are nonbirational.
\end{lem}
\begin{proof}
By
\cite[Proposition 3.6]{MP}
there is an injection
\begin{align*}
\rm{sp}_{\bar{K}, \bfk} \colon \NS X_{\bar{K}} \to \NS X_0 
\end{align*}
whose cokernel is torsion free.
In particular,
$\rm{sp}_{\bar{K}, \bfk}$
is bijective on torsion subgroups.
Then
$\NS_{tor} X_{\bar{K}}, \NS_{tor} X^\prime_{\bar{K}}$
cannot be isomorphic.
\end{proof}

The same argument yields

\begin{lem}
If
$\NS_{tor} X_0, \NS_{tor} X^\prime_0$
are nonisomorphic,
then
$X_{\bar{Q}}, X^\prime_{\bar{Q}}$
are nonbirational.
\end{lem}

If
$\rho (X_0) = \rho (X^\prime_0) = 1$,
then
$X_0, X^\prime_0$
are birational
if and only if
they are isomorphic
\cite[Section 0.5]{BC}.
Indeed,
birational Calabi--Yau manifolds are connected by a sequence of flops
\cite{Kaw}.
However,
no such flops are possible on neither
$X_0$
nor
$X^\prime_0$,
since by
$\rho (X_0) = \rho (X^\prime_0) = 1$
all their nonzero nef divisors are ample.

\begin{lem}
Assume that
$\rho (X_0) = \rho (X^\prime_0) = 1$
and
$\deg (X_0) \neq \deg (X^\prime_0)$.
Then
$X_{\bar{K}}, X^\prime_{\bar{K}}$
are nonbirational.
Here,
the degree is defined with respect to the unique ample generator of the Picard group.
\end{lem}
\begin{proof}
By
\cite[Proposition 3.6]{MP}
we have
$\rho (X_{\bar{K}}) \leq \rho (X^\prime_0) = 1$.
There is an ample divisor
$H$
on
$X_{\bar{K}}$,
as it is a projective variety of positive dimension.
It follows that
$H$
is neither torsion
nor
algebraically equivalent to
$0$.
Indeed,
torsion divisors are numerically trivial
and
one of the two numerically effective divisors is ample
if and only if
so is the other.
Two algebraically equivalent divisors share the degree.
Hence we obtain
$\rho (X_{\bar{K}}) = 1$. 

Recall that
$\deg(X_0)$
is the highest order coefficient of the Hibert polynomial of
$X_0$
multiplied with
$(\dim X_0)!$.
Since
$\pi_R$
is flat projective,
we have
$\deg(X_0) = \deg(X_K)$.
Let
$S(X_K)$
be the homogeneous coordinate ring of
$X_K$
and
$P_{X_K}$
the Hilbert polynomial of
$X_K$.
By definition
$P_{X_K}(l)$
are given by
$\dim_K S(X_K)_l$
for sufficiently large integers
$l \in \bZ$.
Since
$X_{\bar{K}}$
is irreducible,
$\dim_K S(X_K)_l$
is stable under the base change
$X_{\bar{K}} \to X_K$
along algebraic extension
$K \subset \bar{K}$.
Thus we obtain
$\deg(X_K) = \deg(X_{\bar{K}})$
and
$\deg(X_{\bar{K}}) \neq \deg(X^\prime_{\bar{K}})$.
\end{proof}

\begin{lem}
Assume that
$\rho (X_0) = \rho (X^\prime_0) = 1$
and
$\deg (X_0) \neq \deg (X^\prime_0)$.
Then
$X_{\bar{Q}}, X^\prime_{\bar{Q}}$
are nonbirational.
Here,
the degree is defined with respect to the unique ample generator of the Picard group.
\end{lem}
\begin{proof}
The same argument as above works also here.
Assuming that
$\bfk = \bC$,
we show
$\rho(X_{\bar{Q}}) = 1$
in another way.
By
\cite{Ser}
the morphism
$\pi_S \colon X_S \to \Spec S$
corresponds to a proper submersion of complex manifolds
$(\pi_S)_h \colon (X_S)_h \to (\Spec S)_h$. 
Ehresmann lemma tells us that
$(\pi_S)_h$
gives a locally trivial fibration of real manifolds  
\cite{Ehr}.
In particular,
all the fibers of
$(\pi_S)_h$
share the differential type
and 
$\text{H}^2 ((X_s)_h, \bZ)$
is independent from closed points
$s \in \Spec S$.
On the other hand,
we have
$\NS X_s \cong \Pic X_s \cong \text{H}^2 ((X_s)_h, \bZ)$,
as
$X_s$
are Calabi--Yau threefolds.
By
\cite[Theorem 1.1]{MP},
we obtain
$\rho(X_{\bar{Q}}) = \rho(X_0) = 1$.
\end{proof}

In summary,
we obtain

\begin{thm} \label{thm:FMP}
Let
$X_0, X_0^\prime$
be derived-equivalent Calabi--Yau manifolds
of dimension more than two.
Then the geometric generic fibers
$X_{\bar{K}}, X^\prime_{\bar{K}}$
of proper effectivizations
and that
$X_{\bar{Q}}, X^\prime_{\bar{Q}}$
of smooth projective versal deformations
are respectively derived-equivalent Calabi--Yau manifolds.
If,
in addition,
we have either
$\NS_{tor} X_0 \neq \NS_{tor} X^\prime_0$,
or
$\rho (X_0) = \rho (X^\prime_0) = 1$
and
$\deg (X_0) \neq \deg (X^\prime_0)$,
then
they are respectively nonbirational.
Similar results hold for the generic fibers.
\end{thm}

\begin{rmk}
Consider the base changes
$\pi_B \colon X_B \to B, \pi^\prime_B \colon X^\prime_B \to B$
of
$\pi_S, \pi^\prime_S$
to a smooth connected $\bfk$-curve
$B$
containing the point
$s_0$.
The derived equivalence of
$X_S, X^\prime_S$
induces that of
$X_B, X^\prime_B$.
If
$X_0, X^\prime_0$
are nonbirational,
then by
\cite[Theorem 1.1]{KT}
the generic fibers of
$\pi_B, \pi^\prime_B$
must be nonbirational.
Hence the function fields of
$X_B, X^\prime_B$
are nonisomorphic.
Since any field extension gives a faithfully flat module over the original field,
also the geometric generic fibers must be nonbirational. 
\end{rmk}

\subsection{Geometric generic Gross--Popescu pair}
In this subsection,
we temporarily assume that
$\bfk = \bC$.
Now,
we will consider the case of the Gross--Popescu pair
\begin{align*}
X_0 = V^1_{8, y}, \
X^\prime_0
=
V^1_{8, y} / G
\end{align*}
where
$V^1_{8, y}$
is one of the two small resolutions of 
$V_{8, y}$
and
$G = \bZ_8 \times \bZ_8$
freely acts on
$V^1_{8, y}$.
Here,
$V_{8, y}$
is a $3$-dimensional complete intersection in
$\bP^8$
of four hypersurfaces parametrized by a general point
$y \in \bP^2$.
They are derived-equivalent Calabi--Yau threefolds
with
$h^{1,1} (X_0) = h^{1,2} (X_0) = 2$
\cite{GP}.
Since
$X_0$
is simply-connected
\cite[Theorem 1.4]{GP08},
we have
$\text{H}^i (X_0, \scrO_{X_0}) = 0, \ i = 1, 2$.
Although the fundamental group of
$X^\prime_0$
is given by
$G \neq 0$,
we also have
$\text{H}^i (X^\prime_0, \scrO_{X^\prime_0}) = 0, \ i = 1, 2$
by
\cite[Corollary B]{PS}.
Then
$\NS_{tor} X_0 \neq \NS_{tor} X^\prime_0$,
or
in this case equivalently  
$\Pic_{tor} X_0 \neq \Pic_{tor} X^\prime_0$.
Indeed,
from the exponential sequence it follows
$\Pic X_0 = \text{H}^2 ((X_0)_h, \bZ)$.
The torsion part of
$\text{H}^2 ( ( X_0 )_h, \bZ )$
is
$\Tor_1 ( \text{H}_1 ( ( X_0 )_h, \bZ ), \bC )
=
0$
by the universal coefficient theorem
and
Van Kampen's theorem.
On the other hand,
the torsion part of
$\text{H}^2 ( ( X^\prime_0 )_h, \bZ )$
is
$\Tor_1 ( \text{H}_1 ( ( X^\prime_0 )_h, \bZ ), \bC )
=
G$.
Thus one can apply
\pref{thm:FMP}
to see that
$X_{\bar{K}}, X^\prime_{\bar{K}}$  
and
$X_{\bar{Q}}, X^\prime_{\bar{Q}}$  
are respectively
nonbirational derived-equivalent Calabi--Yau threefolds.  

\subsection{Geometric generic Pfaffian--Grassmannian pair}
Next,
we will consider the case of the Pfaffian-Grassmannian pair
\begin{align*}
X_0 = \Gr (2,V_7) \cap \bP (W), \
X^\prime_0
=
\Pf (4,V_7) \cap \bP (W^\perp)
\end{align*}
where
$V_7$
is a $7$-dimensional $\bfk$-vector space,
$W$
is a $14$-dimensional general quotient vector space of
$\wedge^2 V_7 \twoheadrightarrow W$,
and
$W^\perp = \Coker (W^\vee \hookrightarrow \wedge^2 V^\vee_7)$.
They are derived-equivalent Calabi--Yau threefolds
with
$\rho (X_0) = \rho (X^\prime_0) = 1$
and
$\deg (X_0) \neq \deg (X^\prime_0)$
\cite{BC, 0610957}.
One can apply
\pref{thm:FMP}
to see that
$X_{\bar{K}}, X^\prime_{\bar{K}}$  
and
$X_{\bar{Q}}, X^\prime_{\bar{Q}}$  
are respectively
nonbirational derived-equivalent Calabi--Yau threefolds.  
Similarly,
one obtains another example from
Reye congruence and double quintic symmetroid Calabi--Yau threefolds
\cite{HT}.

We will study
$X_{\bar{Q}}, X^\prime_{\bar{Q}}$
slightly further. 
Assume that
$\bfk$
is a universal domain.
Let
$L \subset \bfk$
be a countable subfield
over which
$X_S, X^\prime_S$
and
$\Spec S$
are defined,
i.e.,
we are given $L$-varieties
$X_L, X^\prime_L$
and
$\Spec S_L$
such that
$X_L \times_L \bfk \cong X_S, \ 
X^\prime_L \times_L \bfk \cong X^\prime_S$
and
$\Spec S_L \times_L \bfk \cong \Spec S$.
From the proof of
\cite[Lemma 2.1]{Via}
there exists an isomorphism
$\bfk \to \bar{Q}$
fixing
$L$
such that
the base change of very general fibers
$X_s, X^\prime_s$
of
$X_S, X^\prime_S$
are respectively isomorphic to
$X_{\bar{Q}}, X^\prime_{\bar{Q}}$
as a scheme.
We emphasize that
these isomorphisms of schemes do not fix the base field but
$L$,
different from isomorphisms of varieties.

There is another Fourier--Mukai partner called IMOU varieties
\cite{IMOU, Kuz18},
consisting of derived-equivalent Calabi--Yau threefolds
$Y_0, Y^\prime_0$
which are flat degenerations of 
$X_0, X^\prime_0$
respectively by
\cite[Remark 5.2]{IIM},
\cite[Corollary 6.3]{KK}.
See also
\cite[Introduction]{IMOU}.
Again,
each of
$Y_0, Y^\prime_0$ 
cannot be isomorphic to either of
$X_{\bar{Q}}, X^\prime_{\bar{Q}}$
as a variety,
since their base fields are different.
Thus 
$X_{\bar{Q}}, X^\prime_{\bar{Q}}$
provide a new example of Fourier--Mukai partners,
which
are isomorphic to very general fibers as a scheme
but
have different structures as a variety from both the
Pfaffian--Grassmannian pair
and
IMOU varieties.


\end{document}